\documentclass[12pt,reqno]{amsart}
\usepackage{fullpage}
\usepackage{times}
\usepackage{graphicx}
\usepackage{amsmath}
\usepackage{amsfonts}
\usepackage{amssymb}
\usepackage{amsthm}
\usepackage{mathtools}
\usepackage{xcolor}
\usepackage{comment}
\usepackage{colonequals}
\usepackage{bbm}
\usepackage{enumerate}
\usepackage{bm}
\usepackage{graphicx}
\usepackage{mathrsfs}
\usepackage{fancyvrb}
\usepackage{tikz-cd}
\usetikzlibrary{patterns}
\usepackage[colorlinks=true, pdfstartview=FitH, linkcolor=blue, citecolor=blue, urlcolor=blue]{hyperref}

\allowdisplaybreaks

\newtheorem{theorem}{Theorem}[section]
\newtheorem{lemma}[theorem]{Lemma}
\newtheorem{proposition}[theorem]{Proposition}
\newtheorem{corollary}[theorem]{Corollary}
\newtheorem{conjecture}[theorem]{Conjecture}

\theoremstyle{definition}
\newtheorem{definition}[theorem]{Definition}
\newtheorem{remark}[theorem]{Remark}

\newtheorem{question}[theorem]{Question}

\newcommand{\FF}{\mathbb{F}}
\newcommand{\F}{\mathbb{F}}
\newcommand{\bP}{\mathbb{P}}
\newcommand{\ZZ}{\mathbb{Z}}
\newcommand{\RR}{\mathbb{R}}
\newcommand{\bA}{\mathbb{A}}

\DeclareMathOperator{\Gal}{Gal}

\def\multiset#1#2{\ensuremath{\left(\kern-.3em\left(\genfrac{}{}{0pt}{}{#1}{#2}\right)\kern-.3em\right)}}

\title{Hypersurfaces passing through the Galois orbit of a point}
\author{Shamil Asgarli}
\address{Department of Mathematics \& Computer Science \\ Santa Clara University \\ CA 95053, United States}
\email{sasgarli@scu.edu}
\author{Jonathan Love}
\address{Mathematics Institute\\ Leiden University\\ 2333 CC Leiden\\ Netherlands}
\email{j.r.love@math.leidenuniv.nl}
\author{Chi Hoi Yip}
\address{School of Mathematics\\ Georgia Institute of Technology\\Atlanta, GA 30332\\ United States}
\email{cyip30@gatech.edu}
\subjclass[2020]{Primary 14G15, 14J70; Secondary 14N05, 11G25}
\keywords{hypersurface, finite field, Galois orbit, linear system, irreducibility}
\begin{document}

\begin{abstract}
Asgarli, Ghioca, and Reichstein proved that if $K$ is a field with $|K|>2$, then for any positive integers $d$ and $n$, and separable field extension $L/K$ with degree $m=\binom{n+d}{d}$, there exists a point $P\in \mathbb{P}^n(L)$ which does not lie on any degree $d$ hypersurface defined over $K$. They asked whether the result holds when $|K| = 2$. We answer their question in the affirmative by combining various ideas from arithmetic geometry. More generally, we show that for each positive integer $r$ and separable field extension $L/K$ with degree $r$, there exists a point $P \in \mathbb{P}^n(L)$ such that the vector space of degree $d$ forms over $K$ that vanish at $P$ has the expected dimension. We also discuss applications to linear systems of hypersurfaces with special properties.
\end{abstract} 

\maketitle

\section{Introduction}
Throughout the paper, let $K$ be a field, $L/K$ a separable extension of degree $r$, and $n,d$ positive integers. If $K$ is a finite field with $q$ elements, we write $K=\F_q$. Let $\mathcal{S}_{n,d}(K)$ denote the vector space of homogeneous polynomials over $K$ of degree $d$ in $n+1$ variables, which has dimension $m\colonequals\binom{n+d}{n}$ over $K$.

Each point in $\mathbb{P}^n(\overline{K})$ imposes a linear constraint on the space of degree $d$ forms on $\bP^n$ by requiring the forms to vanish at the given point. We say a set $S$ of $r\leq m$ points is in \emph{general position} if these constraints are linearly independent. 
If $L/K$ is Galois, is there a point  $P\in \bP^n(L)$ such that the $\Gal(L/K)$-orbit of $P$ is in general position? 
Recently, Asgarli, Ghioca, and Reichstein \cite[Theorem 1.1]{AGR24} proved the following result, addressing the case $r=m$. Note that there exists a hypersurface defined over $\overline{K}$ that contains the Galois orbit of $P$ if and only if there exists a hypersurface defined over $K$ that contains $P$.
\begin{theorem}[Asgarli-Ghioca-Reichstein]\label{thm:AGR}
Let $K$ be a field with $|K| > 2$, and let $d, n$ be positive integers. For any separable field extension $L/K$ with degree $m = \binom{n + d}{d}$, there exists a point $P \in \bP^n(L)$ that does not lie on any degree $d$ hypersurface defined over $K$.
\end{theorem}

 As remarked in \cite{AGR24}, Theorem~\ref{thm:AGR} generalizes the primitive element theorem for separable field extensions. 
In this paper, we extend Theorem~\ref{thm:AGR} to cover the remaining case $K=\F_2$ posed as an open question in \cite{AGR24}. We also present a simpler proof of Theorem~\ref{thm:AGR} in the case $K$ is a finite field, and prove the following natural generalization of Theorem~\ref{thm:AGR} for extensions $L/K$ of arbitrary degree $r$.

\begin{theorem}\label{thm:main}
Let $K$ be a field, and $L/K$ be a separable extension of degree $r \geq 1$. For any $n, d \geq 1$, there exists $P \in \bP^n(L)$ such that 
\begin{equation}\label{eq: expected dim}
    \dim_K \{F\in\mathcal{S}_{n,d}(K) \ | \ F(P)=0\}= \max(m-r,0),
\end{equation}
where $m\colonequals\binom{n+d}{n}$ is the dimension of $\mathcal{S}_{n,d}(K)$.
\end{theorem}

In particular, we recover Theorem~\ref{thm:AGR} as a special case when $r=m$ and $|K|>2$. The case $r = m$ and $|K| = 2$, which was left open in \cite{AGR24}, ends up being the most challenging case in the proof of Theorem~\ref{thm:main}. To address this case, we combine various ideas from arithmetic geometry to handle instances where either $d$ or $n$ is sufficiently large; see Section~\ref{sec:proof outline} for a summary. This theoretical approach proves the result for all but finitely many pairs $(n, d)$, which can then be checked explicitly by a computer search.

We first observe that the right-hand side of equation~\eqref{eq: expected dim} is a lower bound for the left-hand side for every $P\in\bP^n(L)$.
Given any such point, the set of polynomials $F\in\mathcal{S}_{n,d}(K)$ satisfying $F(P)=0$ is a subspace of codimension at most $r$. Indeed, we have $\leq r$ linear constraints on $F$ coming from the fact that $F$ must vanish at each of the $\leq r$ Galois conjugates of $P$, defining a subspace of $\mathcal{S}_{n,d}(K)\otimes_K L=\mathcal{S}_{n,d}(L)$ of codimension at most $r$; this space is Galois-invariant and hence descends to a codimension $\leq r$ subspace of $\mathcal{S}_{n,d}(K)$. 
Thus 
\begin{equation*}
    \dim_K \{F\in\mathcal{S}_{n,d}(K) \ | \ F(P)=0\}\geq \max(m-r,0).
\end{equation*}
Theorem~\ref{thm:main} states that the minimal value is always attained by some point $P$.

There are two possible reasons why equation~\eqref{eq: expected dim} may fail for a given $P\in\bP^n(L)$. The first is arithmetic: $P$ may be defined over a subfield of $L$ with degree $r' < r$ over $K$. The second is geometric: Fix any $(\max(m-r,0)+1)$-dimensional subspace of $\mathcal{S}_{n,d}(K)$, and let $V$ denote the common vanishing locus of all degree $d$ forms in this subspace. Then any point in $V(L)$ is, by construction, a point for which equation~\eqref{eq: expected dim} does not hold. Theorem \ref{thm:main} is equivalent to the statement that $\bP^n(L)$ is not contained in the union of all $V(L)$ constructed in this way.

\subsection{Applications to linear systems} Let $\mathbb{P}^n$ be the $n$-dimensional projective space over $K=\F_q$. Let $\mathcal{P}$ be any property that a hypersurface in $\mathbb{P}^n$ may satisfy. For instance, $\mathcal{P}$ might be ``is smooth," ``is irreducible," or ``is geometrically irreducible." This naturally leads to the following question. 

\begin{question}\label{quest:linear-systems}
 What is the maximum (projective) dimension of a linear system $\mathcal{L}$ of hypersurfaces in $\mathbb{P}^n$ such that every $\mathbb{F}_q$-member of $\mathcal{L}$ satisfies property $\mathcal{P}$?    
\end{question}

Question~\ref{quest:linear-systems} has been studied in various settings; see, for example, \cite{AGR23}, \cite{AGR24}, \cite{AGY23} for the cases when $\mathcal{P}$ represents the property of being smooth, irreducible, non-blocking, respectively. We phrase Question~\ref{quest:linear-systems} in concrete terms. We want to determine the maximum value of an integer $t$ such that there exist polynomials $F_0, F_1, ..., F_{t}$ in $n+1$ homogeneous variables such that the hypersurface $X_{[a_0:a_1:\cdots:a_t]}$ defined by the equation $a_0 F_0 + a_1 F_1 + \cdots + a_t F_{t}=0$ satisfies property $\mathcal{P}$ for \emph{every} choice $[a_0:a_1:\cdots:a_t]\in\mathbb{P}^{t}(\mathbb{F}_q)$. In this case, the desired linear system is $\mathcal{L}=\langle F_0, ..., F_t\rangle\cong \mathbb{P}^{t}$. 

When $\mathcal{P}$ denotes ``is irreducible over $\mathbb{F}_q$", Asgarli, Ghioca and Reichstein \cite[Theorem 1.3]{AGR24} answered Question~\ref{quest:linear-systems}: the maximum (projective) dimension of a linear system $\mathcal{L}$ of degree $d$ hypersurfaces where each $\mathbb{F}_q$-member is irreducible over $\mathbb{F}_q$ is $\binom{n+d}{n}-\binom{n+d-1}{n}-1$. We generalize this result by weakening the irreducibility requirement to allow each $\mathbb{F}_q$-member to contain an irreducible component of a large degree.

\begin{theorem}\label{thm:system-prescribed-degree} Let $d\geq 2$ and $2\leq i\leq d$. There exists a linear system $\mathcal{L}$ of degree $d$ hypersurfaces in $\bP^n/\F_q$ with (projective) dimension equal to
$\binom{n+d}{n}-\binom{n+i-1}{n}-1$
such that each $\mathbb{F}_q$-member of $\mathcal{L}$ has an $\mathbb{F}_q$-irreducible component of degree at least $i$. Moreover, the result is sharp: $\dim(\mathcal{L})$ cannot be increased to $\binom{n+d}{n}-\binom{n+i-1}{n}$.
\end{theorem}

Section~\ref{sec:linear-systems} presents a more general result (Theorem~\ref{thm:system-prescribed-factors}) that further extends Theorem~\ref{thm:system-prescribed-degree}. We include Theorem~\ref{thm:system-prescribed-degree} here in the introduction to motivate our results, as it is simpler to present and illustrates how our work generalizes the corresponding result in \cite{AGR24}. The same bound holds if we replace $\F_q$ with a number field, but not if we replace $\F_q$ with an arbitrary field; see Remark~\ref{rmk:replace Fq}. We also find an exact answer to Question~\ref{quest:linear-systems} when $\mathcal{P}$ stands for ``is reduced" (see Corollary~\ref{cor:linear-systems-reduced}).

\subsection{Proof outline of Theorem~\ref{thm:main}}\label{sec:proof outline}

We prove Theorem~\ref{thm:main} for infinite fields $K$ in Section~\ref{sec:infinite}, following the method in \cite{AGR24}. For the rest of this proof outline, suppose $K=\F_q$ for $q\geq 2$ a prime power, so $L=\F_{q^r}$. In Section \ref{sec:finite setup}, we introduce some of the tools that will be used throughout the proof of the finite field case, including a reduction to the case $n\geq 2$, $d\geq 2$, and $r>\binom{n-1+d}{n-1}$ in Section \ref{sec:special cases reductions}.

The first method we use to study Theorem \ref{thm:main} in the finite field case is to count incidences between points and hypersurfaces (Section~\ref{sec:incidence}). More precisely, we count pairs $(P,H)$, where $P\in\bP^n(\F_{q^r})$ lies on a degree $d$ hypersurface $H$, in two different ways. First, we bound the number of points on each hypersurface and sum this bound over all hypersurfaces.  Second, we count the hypersurfaces passing through each point and add up this count over all points. Since these two counts must agree, it follows that not too many points can lie on a number of hypersurfaces that is larger than expected. This argument is carried out in the proof of Proposition~\ref{prop:less than q}. Using this bound, we prove Theorem~\ref{thm:main} for all but finitely many cases with $q \geq 3$, as well as for all but finitely many cases with $q = 2$ and $r \neq m$; this is carried out in Section~\ref{subsec: checking inequality}. The remaining exceptional cases are verified explicitly in Appendix~\ref{sec:computer}.

Unfortunately, Proposition~\ref{prop:less than q} is unhelpful in the case $q=2$ and $r=m$.
To obtain a more refined bound, we account for the points lying in the intersection between distinct hypersurfaces. Depending on the relative size of the parameters $n$ and $d$, we apply different methods to understand the structure of these intersections. The proof strategy needed for different values of $n$ and $d$ is summarized in Figure~\ref{fig:r=m cases} below.

\begin{figure}[h!]
    \centering
    \begin{tikzpicture}[scale=0.47]
    \def\rows{12}
    \def\cols{16}

    \fill[pattern=north west lines] (1,13) rectangle ++ (18,-1);
    \fill[pattern=north east lines] (1,13) rectangle ++ (1,-12);
    \fill[pattern=horizontal lines] (4,12) rectangle ++ (15,-1);
    \fill[pattern=horizontal lines] (2,8) -- (6,8) -- (6,7) -- (7,7) -- (7,6) -- (8,6) -- (8,5) -- (9,5 ) -- (9, 4) -- (10,4) -- (10,3) -- (11,3) -- (11,2) -- (12,2) -- (12,1) -- (2,1) -- cycle;
    \fill[pattern=vertical lines] (19,11) -- (17,11) -- (17,9) -- (13, 9) -- (13,7) -- (11, 7) -- (11,5) -- (10,5) -- (10,1) -- (19,1) -- cycle;

    \draw[black] (1, 1.5) -- (1, \rows + 1) -- (18.5, \rows + 1);

    \node at (15, \rows + 1.5) {$\cdots$};
    \node at (16.5, \rows + 1.5) {35};
    \node at (17.5, \rows + 1.5) {36};
    \node at (18.5, \rows + 1.5) {$\cdots$};

    \node at (9,9) {Appendix \ref{sec:computer}};
    \node[anchor=west] at (19, \rows + .5) {$\leftarrow$ Remark \ref{rmk: simpler cases}};
    \node[anchor=west] at (19, \rows - .5) {$\leftarrow$ Theorem \ref{thm:q=d=2}};
    \node[anchor=west] at (19, 6) {$\leftarrow$ Theorem \ref{thm:inc-exc}};
    \node[anchor=north, align=center] at (1.5, 1) {$\uparrow$ \\ Section \ref{sec:special cases reductions}};
    \node[anchor=north, align=center] at (7, 1) {$\uparrow$ \\ Theorem \ref{thm:q 2 case}};

    \node at (10, 14.5) {$n$};
    \node at (-.5, 7) {$d$};
    \foreach \x in {1,...,13} {
        \node at (\x + 0.5, \rows + 1.5) {\x};
    }
    \foreach \y in {1,...,11} {
        \node at (0.5, 13.5 - \y) {\y};
    }
    \node at (0.5, 1.5) {$\vdots$};
    \end{tikzpicture}
    \caption{How to prove Theorem \ref{thm:main} for $K=\F_2$, each given $(n,d)$, and $r=m=\binom{n+d}{n}$.}
    \label{fig:r=m cases}
\end{figure}

To understand the difficulty, observe that when $r=m$, it suffices to count the $\F_{q^m}$-points lying on the union of all hypersurfaces over $\F_q$ and show that the total is less than the number of points in $\bP^n(\F_{q^m})$. If the average degree $d$ hypersurface contains $q^{m(n-1)}(1 + \delta)$ points for some average error term $\delta$, and we sum this bound over all $\frac{q^m - 1}{q - 1}$ degree $d$ hypersurfaces, we would need to prove that:
\begin{equation}\label{eq: summary to prove}
    q^{m(n-1)}(1+\delta)\left(\frac{q^m-1}{q-1}\right)\leq \frac{q^{m(n+1)}-1}{q^m-1},
\end{equation}
which implies
\[1+\delta\leq (q-1)+O(q^{-m}).\]
When $q = 2$, the inequality fails unless $\delta$ is bounded by a constant multiple of $2^{-m}$. Proving such a strong bound on the average number of $\F_{q^m}$-points over the set of hypersurfaces of degree $d\geq 2$ over $\F_q$ seems to be beyond reach. Instead, we use alternate methods that account for points lying in intersections of hypersurfaces.

Our second method, discussed in Section~\ref{sec:avoid}, focuses on counting points on irreducible components rather than the full (possibly reducible) hypersurfaces. This allows us to sharpen the error term $\delta$, as point-counting bounds for irreducible varieties are generally stronger than those for general projective varieties. Moreover, it significantly reduces the number of hypersurfaces to consider, as each irreducible hypersurface of degree $e< d$ occurs as a component in a very large number of degree $d$ hypersurfaces. This reduction decreases the left-hand side of inequality~\eqref{eq: summary to prove}, establishing the desired bound when $d$ is sufficiently large compared to $n$ (Theorem \ref{thm:q 2 case}). The approach also works for $d = 2$, as it yields extremely sharp bounds on $\delta$ in this case (Theorem~\ref{thm:q=d=2}). This approach also provides a much shorter proof of Theorem~\ref{thm:AGR} when $K$ is finite.

The third method, discussed in Section~\ref{sec:inc-exc}, applies the inclusion-exclusion principle. By adding the points on each hypersurface, subtracting those on pairwise intersections, and adding those on triple intersections, we obtain an upper bound on the total number of points in the union of hypersurfaces. For this method to be effective, we need a relatively strong upper bound on the average number of points on the intersections of three hypersurfaces. Specifically, we must show that ``most" triple intersections are irreducible, as we have much stronger point-counting bounds for irreducible varieties. If a variety is reducible, then the intersection of two of its components forms a locus of singular points with large dimension. So, to bound the number of triple intersections that are reducible, it suffices to bound the number of triple intersections with large singular locus.
We achieve this by adapting an argument due to Poonen~\cite{P04}. The detailed analysis is carried out in Section~\ref{sec:singular locus}. Using this bound, which is valid only when $d \geq 3$ and $n$ is sufficiently large, we get a nontrivial upper bound on the number of points in the union of degree $d$ hypersurfaces (Theorem~\ref{thm:inc-exc}).

\medskip

\textbf{Organization of the paper.}
In Section~\ref{sec:infinite}, we prove Theorem~\ref{thm:main} for infinite fields $K$. In Section~\ref{sec:finite setup}, we prove several preliminary estimates over finite fields. We then apply the three methods outlined above to prove Theorem~\ref{thm:main} for finite fields in Sections~\ref{sec:incidence},~\ref{sec:avoid}, and~\ref{sec:inc-exc}. We discuss some applications of our main result to linear systems in Section~\ref{sec:linear-systems}. 
Finally, in Section~\ref{sec:conj}, we discuss a result and a conjecture related to the number of points $P$ satisfying the conclusion of Theorem~\ref{thm:main}. 

\section{Proof for infinite fields}\label{sec:infinite}

We structure our proof following \cite[Section 2]{AGR24}, which handles the case $r = m \colonequals \binom{n + d}{n}$.
We begin by generalizing \cite[Lemma 2.1]{AGR24}. 

\begin{lemma}\label{lemma:infinite-field} Let $K$ be an infinite field and suppose $r\in\mathbb{N}$ and $m=\binom{n+d}{n}$. There exist points $P_1, ..., P_r\in\mathbb{P}^n(K)$ such that 
$$
\dim_K \{F\in \mathcal{S}_{n, d}(K) \ | \ F(P_i)=0 \text{ for each } 1\leq i\leq r\} = \max(m-r,0).
$$
\end{lemma}

\begin{proof} 
If $r>m$, then the lemma follows immediately from the case $r=m$, so we may assume $r\leq m$.
We pick the points $P_1, \ldots, P_r$ inductively to ensure:
\begin{equation}\label{eq:infinite-field-independence}
\dim_K \{F\in \mathcal{S}_{n, d}(K) \ | \ F(P_i)=0 \text{ for each } 1\leq i\leq j\} = m-j
\end{equation}
for each $1\leq j\leq r$. Choose $P_1\in \mathbb{P}^n(K)$ arbitrarily. The condition $F(P_1)=0$ imposes exactly one linear condition, so equation~\eqref{eq:infinite-field-independence} holds for $j=1$. For $1\leq j<r$, suppose $P_1, \ldots, P_{j}$ are chosen according to equation~\eqref{eq:infinite-field-independence}. Pick a nonzero $F\in \mathcal{S}_{n, d}(K)$ satisfying $F(P_i)=0$ for $1\leq i\leq j$; such an $F$ exists because $m-j>0$. Since $K$ is infinite, there exists $P_{j+1}\in \mathbb{P}^n(K)$ such that $F(P_{j+1})\neq 0$. Then $P_1, ..., P_{j+1}$ satisfy the desired equality \eqref{eq:infinite-field-independence}. 
\end{proof}

Next, we generalize \cite[Proposition 2.2]{AGR24} to our setting. 

\begin{proposition}\label{prop:existence-minors}
Let $L$ be a commutative algebra over a field $K$, and fix an isomorphism $L\simeq K^r$ as vector spaces over $K$. Let $n,d\geq 1$ and $m=\binom{n+d}{n}$. Then there exist homogeneous polynomial functions $H_1, ..., H_t$ on $\mathbb{A}_K^{r(n+1)}$ satisfying the following condition: if $K'/K$ is any field extension, $L'\colonequals L\otimes_K K'$, and $P\in \mathbb{A}_K^{n+1}(L')$, we have
\begin{align}\label{eq:existence-minors}
\dim_{K'} \{F\in \mathcal{S}_{n, d}(K') \ | \  F(P)=0  \} > \max(m - r, 0)
\end{align}
if and only if $H_i(\phi(P))=0$ for each $1\leq i\leq t$, where $\phi\colon \bA_K^{n+1}(L')\to \bA_K^{r(n+1)}(K')$ is the isomorphism induced by the fixed isomorphism $L\to K^r$.
\end{proposition}

\begin{proof} Denote by $M_1, \ldots, M_{m}$ the distinct monomials of degree $d$ in $x_0, \ldots, x_n$. We will show that a given point $P= (a_0,\ldots,a_n)\in \mathbb{A}_K^{n+1}(L')$ satisfies inequality~\eqref{eq:existence-minors} if and only if a certain $m\times r$ matrix has vanishing minors. To define this matrix, observe that the isomorphism $L\simeq K^r$ determines a basis $b_1,\ldots,b_r$ for $L$ over $K$, so for each $a_i\in L'$ we can write $a_i = y_{i, 1} b_1 + \cdots + y_{i, r} b_r$ for some $y_{i, j}\in K'$; the map $P\mapsto (y_{i,j})_{0\leq i\leq n, 1\leq j\leq r}$ defines the isomorphism $\phi$. Since $L$ is a $K$-algebra, each product $b_i b_j$ can be expressed as a $K$-linear combination of $b_1, ..., b_r$. Hence, for each $1\leq s\leq m$, we can express 
$$
M_s(P) = \eta_{s,1}(\phi(P)) b_1 + \cdots + \eta_{s, r}(\phi(P)) b_r
$$
where each $\eta_{s,k}(y_{i,j})$ is a homogeneous polynomial of degree $d$ in $y_{i,j}$ with coefficients in $K$, viewing $y_{i,j}$ as indeterminates for $0\leq i\leq n$ and $1\leq j\leq r$. Consider the $m\times r$ matrix
$$
U(y_{i,j}) = \begin{pmatrix} \eta_{1, 1}(y_{i,j}) & \eta_{1, 2}(y_{i,j}) & \cdots &\eta_{1, r}(y_{i,j}) \\ 
\eta_{2, 1}(y_{i,j}) & \eta_{2, 2}(y_{i,j}) & \cdots & \eta_{2, r}(y_{i,j}) \\
\vdots & \vdots & \ddots & \vdots \\ 
\eta_{m, 1}(y_{i,j}) & \eta_{m, 2}(y_{i,j}) & \cdots & \eta_{m, r}(y_{i,j}) \\
\end{pmatrix}.
$$
For our homogeneous functions $H_1,\ldots,H_t$ on $\mathbb{A}_K^{(n+1)r}$ we take all maximal minors of $U$: all $r\times r$ minors if $r\leq m$, and all $m\times m$ minors if $r>m$.

Now each $F\in \mathcal{S}_{n,d}(K')$ is a $K'$-linear combination of the monomials $M_i$, so there is a vector $(c_1,\ldots,c_m)\in (K')^m$ such that $F(P)=c_1 M_1(P)+\cdots +c_mM_m(P)$ for any $P\in \mathbb{A}_K^{n+1}(L')$. We therefore have $F(P)=0$ if and only if $(c_1,\ldots,c_m)U(\phi(P))=(0,\ldots,0)$, that is, $(c_1,\ldots,c_m)$ is in the kernel of $v\mapsto v U(\phi(P))$. If $U(\phi(P))$ has maximal rank $\min(m,r)$, then the dimension of this kernel is $m-\min(m,r)=\max(m-r,0)$. The kernel has greater dimension if and only if all the maximal minors vanish.
\end{proof}

We are now ready to prove our main theorem over infinite base fields.

\begin{proof}[Proof of Theorem~\ref{thm:main} when $K$ is infinite] Fix a basis for $L$ as a $K$-vector space, and let $H_1, H_2, ..., H_t$ be the homogeneous functions from Proposition~\ref{prop:existence-minors}. We will prove that at least one of these functions is not identically zero. By Lemma~\ref{lemma:infinite-field} applied to the algebraic closure $\overline{K}$ of $K$, there exist $P_1,\ldots,P_r\in\bA^{n+1}(\overline{K})$, none equal to the zero point, satisfying
\[\dim_{\overline{K}}\{F\in \mathcal{S}_{n,d}(\overline{K}) \ | \ F(P_j)=0\text{ for all }1\leq j\leq r\}=\max(m-r,0).\]
Since $L/K$ is separable, there exists an isomorphism of  $\overline{K}$-algebras from $L'=L\otimes_{K} \overline{K}$ to the $r$-fold direct product $\overline{K}\times \cdots \times \overline{K}$. Hence, there is a bijective correspondence between points $P\in \bA^{n+1}(L')$ and $r$-tuples of points $(P_1,\ldots,P_r)\in \bA^{n+1}(\overline{K})^r$, such that for $F\in \mathcal{S}_{n,d}(\overline{K})$ we have $F(P)=0$ if and only if $F(P_j)=0$ for all $1\leq j\leq r$. We therefore obtain a nonzero point $P\in \bA^{n+1}(L')$ satisfying
\[\dim_{\overline{K}}\{F\in \mathcal{S}_{n,d}(\overline{K}) \ | \ F(P)=0\}=\max(m-r,0).\]
By Proposition~\ref{prop:existence-minors} applied to $K'=\overline{K}$, we have $H_{\ell}(\phi(P))\neq 0$ for some $\ell$. This proves $H_\ell\neq 0$.

Since $K$ is infinite, we can find $(y_{i,j})\in \bA_K^{r(n+1)}(K)$ such that $H_\ell(y_{i,j})\neq 0$. 
So, by Proposition~\ref{prop:existence-minors} with $K'=K$, we have a point $\phi^{-1}(y_{i,j})\in \bA_K^{n+1}(L)$ that satisfies equation~\eqref{eq: expected dim}.
\end{proof}

\section{Preliminaries and setup for finite fields}\label{sec:finite setup}

 Let $q\geq 2$ be a prime power. From now on, we consider the case when $K = \FF_q$ is a finite field with $q$ elements. Let $n,d,r\geq 1$, and $m\colonequals\binom{n+d}{n}$. 

Let $\mathcal{S}_{n,d}=\mathcal{S}_{n,d}(\FF_q)$ denote the affine space of homogeneous degree $d$ polynomials in $n+1$ variables over $\FF_q$. Define
\begin{equation}\label{eq:mudef}
\mu(q,n,d,r)=\frac{\#\{P\in\bP^n(\FF_{q^r}) \ | \ \dim_{\FF_q}\{F\in \mathcal{S}_{n,d} \ | \ F(P)=0\}= \max(m-r,0)\}}{\#\bP^n(\FF_{q^r})}
\end{equation}
to be the proportion of $\FF_{q^r}$-points for which equation~\eqref{eq: expected dim} holds. Our main objective will be to derive a positive lower bound for $\mu$.

\subsection{Special cases and reductions}\label{sec:special cases reductions}

We first consider some simple cases. If $r=1$, then for all $q,n,d$, we have $\mu(q,n,d,1)=1$ since for every $P\in\bP^n(\FF_q)$ there exists $F\in\mathcal{S}_{n,d}$ that does not vanish at $P$. Henceforth, we assume $r\geq 2$. 

\begin{lemma}\label{lem:n=1}
For $r\geq 2$ and $n=1$ we have
\begin{align*}
    \mu(q,1,d,r)&= \frac{\#\{\theta\in \F_{q^r} \ | \ \theta\notin \F_{q^k}\text{ for every }k<\min(r,d+1)\}}{\#\bP^1(\F_{q^r})}.
\end{align*}    
\end{lemma}
\begin{proof}
The set of $F \in \mathcal{S}_{1, d}$ that vanish at $[1 : 0]$ has dimension $m - 1$ (where $m = d + 1$), not the expected $m - r$. Thus, it suffices to consider points of the form $P = [\theta : 1]$ for $\theta \in \mathbb{F}_{q^r}$. Let $k\leq r$ denote the degree of the minimal polynomial of $\theta$ over $\F_q$. Then   $F\in\mathcal{S}_{1,d}$ vanishes at $P$ if and only if its evaluation at $(x,1)$ is a multiple of the minimal polynomial of $\theta$. The set of all such $F$ has dimension $\max(m-k,0)$, and we have the expected dimension if and only if this equals $\max(m-r,0)$. Thus, the points of interest are those $P=[\theta:1]$ for $\theta\in\F_{q^r}$ such that the degree $k$ of $\theta$ over $\F_q$ satisfies  $\max(m-k,0)=\max(m-r,0)$, that is, $k \geq \min (r,m)=\min(r,d+1)$, as desired.
\end{proof}

Since there exists a primitive root in $\F_{q^r}$, it follows from Lemma~\ref{lem:n=1} that $\mu(q,1,d,r)>0$. Thus, from now on, we assume $n\geq 2$. It is also not hard to show that $\mu(q,n,1,r)>0$ in the case $d=1$. We postpone the argument to Remark \ref{rmk: simpler cases} simply because the method fits in well with the next section, but the proof does not logically depend on any ensuing results. Unless stated otherwise, we assume $d\geq 2$ for the remainder of the paper.

Finally, the following result allows us to reduce the number of dimensions $n$ we need to check for any fixed $q,d,r$.

\begin{lemma}\label{lem:reduce-n}
    For a prime power $q\geq 2$ and positive integers $n\geq n'\geq 1$, $d\geq 1$, and $1\leq r\leq m' \colonequals\binom{n'+d}{n'}$, if $\mu(q,n',d,r)>0$ then $\mu(q,n,d,r)>0$.
\end{lemma}
\begin{proof}
Let $\mathcal{S}_{n,d}$ and $\mathcal{S}_{n',d}$ denote the space of homogeneous degree $d$ polynomials over $\F_{q}$ in $n+1$ (respectively $n'+1$) variables.
    Pick $P' \in \bP^{n'}(\F_{q^r})$ such that 
    $\dim_{\F_q} \{F\in\mathcal{S}_{n',d} \ | \ F(P')=0\}= m'-r.$
    Let $P \in \bP^{n}(\F_{q^r})$ be the point such that the first $n'+1$ coordinates are the same as $P'$, and all remaining coordinates equal to $0$. We have a linear map $\psi\colon\mathcal{S}_{n,d}\to \mathcal{S}_{n',d}$ obtained by simply dropping all monomials involving variables beyond the first $n'+1$ variables, and $F(P)=0$ if and only if $\psi(F)(P')=0$. Since $\psi$ is surjective, the space of functions in $\mathcal{S}_{n,d}$ vanishing at $P$ has the same codimension as the space of functions in $\mathcal{S}_{n',d}$ vanishing at $P'$, which equals $r$ since $r\leq m'$. Thus, $P$ is a point such that equation~\eqref{eq: expected dim} holds, that is, $\mu(q,n,d,r)>0$.
\end{proof}

Now suppose $r\leq \binom{n-1+d}{n-1}$, and let $n'$ be the unique positive integer satisfying $\binom{n'-1+d}{n'-1}<r\leq \binom{n'+d}{n'}$; then $n'\leq n-1$. If we can prove $\mu(q,n',d,r)>0$, then we have $\mu(q,n,d,r)>0$ by Lemma \ref{lem:reduce-n}. Thus, without loss of generality, we may reduce to the case 
\begin{equation}
    r>\binom{n-1+d}{n-1}.
\end{equation}
In particular, since $n\geq 2$ we may reduce to the case $r\geq d+2$. In fact, for all $r< d+2$, Theorem \ref{thm:main} follows immediately by an interpolation argument (see for example \cite[Lemma 2.2]{AGY23}).

\subsection{A useful lemma from calculus}

At many points in the discussion below, we consider functions of the form
\begin{equation}\label{eq:poly exp form}
    f(q,t)=\sum_{i=1}^k f_i(t)q^{-g_i(t)},
\end{equation}
for polynomials $f_i$ and $g_i$ with positive leading coefficient and $q\geq 2$. We will frequently state without proof an upper bound on $f(q,t)$ that holds for all $q\geq q_0$ and $t\geq t_0$. These claims can be justified by a finite computation using the following lemma.

\begin{lemma}\label{lem: poly exp}
    Let $f(q,t)$ be as above, $M\in\RR$, and $t_0,z,q_0\in\ZZ$ with $t_0\leq z$ and $q_0>1$. Suppose that:
    \begin{itemize}
        \item $f_i(t),g_i(t)\geq 0$ for all integers $t\geq t_0$ and $i=1,\ldots, k$;
        \item $f_i(x)g_i'(x)\log q_0\geq f_i'(x)$ for all real $x\geq z$ and $i=1,\ldots, k$;
        \item $f(q_0,t)<M$ for all integers $t_0\leq t\leq z$.
    \end{itemize}
    Then $f(q,t)<M$ for all integers $t\geq t_0$ and $q\geq q_0$.
\end{lemma}
\begin{proof}
    The derivative of $f(q_0,x)$ with respect to $x$ is 
    \[\sum_{i=1}^k (f_i'(x)-f_i(x)g_i'(x)\log q_0)q_0^{-g_i(x)},\]
    which by assumption is non-positive for all $x\geq z$. Thus $f(q_0,x)\leq f(q_0,z)<M$ for all $x\geq z$, so in fact we have $f(q_0,t)<M$ for all integers $t\geq t_0$. For a fixed $t\geq t_0$, the function $q\mapsto f(q, t)$ is decreasing. Thus $f(q,t)\leq f(q_0,t)<M$ for all $q\geq q_0$.
\end{proof}

\subsection{Bounds on the number of rational points on a hypersurface}\label{sec:LW}
We will use an explicit version of the Lang-Weil bound \cite{LW54} due to Cafure and Matera \cite[Theorem 5.2]{CM06}. For convenience we state a more general bound that holds for all irreducible hypersurfaces regardless of whether or not they are geometrically irreducible. Define
$$
\Delta(q,d)\colonequals(d-1)(d-2)q^{-1/2}+5d^{13/3}q^{-1}.
$$
\begin{lemma} \label{lem:LW}
If $X \subseteq \bP^n$ is a degree $d$ irreducible hypersurface over $\F_q$, then 
$$
\# X(\F_{q}) \leq \frac{q^n-1+q^n\Delta(q,d)}{q-1}.
$$
\end{lemma}\begin{proof}
If $X$ is geometrically irreducible, the bound follows immediately by applying \cite[Theorem 5.2]{CM06} to the affine cone over $X$ in $\bA^{n+1}$ (that is, the affine variety obtained by taking the homogeneous polynomial defining $X$ and considering its vanishing locus in $\bA^{n+1}$). Otherwise, if $X$ is geometrically reducible, we apply \cite[Lemma~2.3]{CM06} to the affine cone over $X$, giving the bound
\[
    \# X(\F_q) \leq d^2\frac{q^{n-1}-1}{q-1}\leq \frac{q^nd^{13/3}q^{-1}}{q-1}\leq \frac{q^n\Delta(q,d)}{q-1}.\qedhere
\]
\end{proof}

We will also need both an upper and lower bound for geometrically irreducible varieties that are not necessarily hypersurfaces. The version below is an immediate consequence of \cite[Theorem 7.1]{CM06} but only holds for sufficiently large $q$. 
\begin{lemma}\label{lem:LW not hypersurface}
    Let $X\subseteq\bP^n$ be a geometrically irreducible variety over $\F_q$ of dimension $\ell\geq 1$ and degree $d$. If $q > 2(\ell + 1)d^2$, then
    \[\left|\#X(\F_q)-\frac{q^{\ell+1}-1}{q-1}\right|\leq\frac{q^{\ell+1}\Delta(q,d)}{q-1}.\]
\end{lemma}
For varieties that are not geometrically irreducible, we will use a considerably weaker upper bound on the number of points.
A variety is \emph{equidimensional} if every irreducible component has the same dimension. In this setting, we have a result due to Couvreur~\cite[Corollary 3.3]{C16}, which states the following in the case $\ell\geq n/2$.
\begin{lemma}\label{lem:reducible bound}
    If $X\subseteq\bP^n$ is an equidimensional projective variety over $\F_q$ of dimension $\ell\geq\frac{n}{2}$ and degree $d$, then
    \[\#X(\F_q)\leq \frac{q^{\ell+1}-1}{q-1}+(d-1)\left(\frac{q^{\ell+1}-1}{q-1}-\frac{q^{2\ell-n+1}-1}{q-1}\right).\]
\end{lemma}
\noindent In the special case $\ell=n-1$ this reduces to
\[\#X(\F_q)\leq \frac{q^n-1}{q-1}+(d-1)q^{n-1},\]
a result originally due to Serre \cite{S91} and proven independently by S\o rensen \cite{S94}. The only other special case we will need is $\ell=n-3$, in which case we have
\[\#X(\F_q)\leq \frac{q^{n-2}-1}{q-1}+(d-1)(q^{n-3}+q^{n-4}+q^{n-5}).\]

\subsection{Bounds on the number of reducible hypersurfaces}\label{sec:bound-reducible}

In the inequalities that follow we will need bounds on the number of points on a hypersurface $H$. As we saw in Section~\ref{sec:LW} above, the available bounds are much stronger when $H$ is irreducible. To achieve the desired results, we will show that ``most'' $H$ are irreducible. 

Let $\mathcal{R}_{n,d}\subseteq \mathcal{S}_{n,d}\setminus\{0\}$ denote the set of polynomials that are reducible over $\F_q$, and set 
\[t\colonequals t(q,n,d)=\frac{\#\mathcal{R}_{n,d}}{\#(\mathcal{S}_{n,d}\setminus\{0\})}.\]
As in the proof of \cite[Proposition 2.7]{P04}, we observe that every element of $\mathcal{R}_{n,d}$ can be written as a product of a degree $i$ polynomial and a degree $d-i$ polynomial for some $1\leq i\leq \frac{d}{2}$, so that
\begin{equation}\label{eq:t:bound}
    t\leq\frac{\#(\mathcal{R}_{n,d}\cup\{0\})}{\#\mathcal{S}_{n,d}}\leq \frac{1}{\#\mathcal{S}_{n,d}}\sum_{i=1}^{\lfloor d/2\rfloor}\left(\#\mathcal{S}_{n,i}\right)\left(\#\mathcal{S}_{n,d-i}\right)=\sum_{i=1}^{\lfloor d/2\rfloor} q^{-N_i}
\end{equation}
for $N_i:=N_i(n,d)=\binom{n+d}{n}-\binom{n+i}{n}-\binom{n+d-i}{n}$. 
We use this to prove a bound similar to \cite[Proposition 3.2]{AGR24}; note that the value of $t$ there is larger than ours, as their count also includes hypersurfaces that are irreducible over $\F_q$ but not geometrically irreducible. 

\begin{lemma}\label{lem:bound t}
    Let $q\geq 2$ be a prime power. If $n\geq 3$, or if $n=2$ and $d\geq 7$, then $t(d-1)\leq \frac1{2q}$.
\end{lemma}
    
\begin{proof}
If $(n,d)=(2,7),(2,8),(2,9)$, we have
\begin{align*}
    (d-1)qt&\leq 6q(q^{-5}+q^{-9}+q^{-11}),\\
    (d-1)qt&\leq 7q(q^{-6}+q^{-11}+q^{-14}+q^{-15}),\\
    (d-1)qt&\leq 8q(q^{-7}+q^{-13}+q^{-17}+q^{-19}),
\end{align*}
respectively. These upper bounds are less than $\frac12$ for all $q\geq 2$. It suffices to prove the result assuming $n=2$ and $d\geq 10$, or assuming $n\geq 3$.

For fixed $d$ and $1\leq i\leq \lfloor d/2\rfloor$, we claim that $N_i(n,d)$ is an increasing function of $n$. Indeed, since the function $\binom{x}{n+1}$ is convex for $x\geq n$, we have $$\binom{n+d}{n+1}-\binom{n+d-i}{n+1}\geq \binom{n+i}{n+1}-\binom{n}{n+1}=\binom{n+i}{n+1}.$$ It follows that
\begin{align*}
   &\left(\binom{n+d+1}{n+1}-\binom{n+i+1}{n+1}-\binom{n+d-i+1}{n+1}\right)-\left(\binom{n+d}{n}-\binom{n+i}{n}-\binom{n+d-i}{n}\right)\\
    &= \binom{n+d}{n+1}-\binom{n+i}{n+1}-\binom{n+d-i}{n+1}\geq 0.
\end{align*}
Thus for $n\geq 3$, we have 
\begin{align*}
    N_i &\geq \binom{d+3}{3}-\binom{i+3}{3}-\binom{d-i+3}{3}\\
     &=\frac12 (d + 4) i (d - i)-1 \geq \frac12(d+4)(d-1)-1
\end{align*}
for each $1\leq i\leq d/2$. Using inequality~\eqref{eq:t:bound}, we deduce that
\begin{align*}
        (d-1)qt\leq (d-1)q\left(\sum_{i=1}^{\lfloor d/2\rfloor}q^{-N_i}\right)
        \leq \frac12d(d-1)q^{2-(d+4)(d-1)/2},
    \end{align*}
    which by Lemma \ref{lem: poly exp} is less than or equal to $\frac12$ for all $q\geq 2$ and $d\geq 2$.

    If $n=2$, we instead have
    \begin{align*}
        N_i &= \binom{d+2}{2}-\binom{i+2}{2}-\binom{d-i+2}{2}=i(d-i)-1 \geq d-2
    \end{align*}
    for each $1\leq i\leq d/2$. Inequality~\eqref{eq:t:bound} implies that
    \begin{align*}
        (d-1)qt&\leq \frac12d(d-1)q^{3-d},
    \end{align*}
    which by Lemma \ref{lem: poly exp} is less than $\frac12$ for all $q\geq 2$ and $d\geq 10$.
\end{proof}

We also note the following weaker bound that holds more generally.

\begin{lemma}\label{lem:bound t 2}
    Let $q\geq 2$ be a prime power, $n\geq 2$, and $d\geq 1$. Then $t(d-1)\leq \frac 98$.
\end{lemma}
\begin{proof}
    If $n\geq 3$, or if $n=2$ and $d\geq 7$, this is immediate from Lemma~\ref{lem:bound t}. For $d=1$ the bound is trivial, and for $d=2$ it follows because $t\leq 1$. It suffices to verify the claimed inequality for $n=2$ and $3\leq d\leq 6$. Considering each value of $d$ in turn, by inequality~\eqref{eq:t:bound} we obtain 
    \begin{align*}
        d=3:t(d-1)&\leq 2q^{-1},\\
        d=4:t(d-1)&\leq 3q^{-2}+3q^{-3},\\
        d=5:t(d-1)&\leq 4q^{-3}+4q^{-5},\\
        d=6:t(d-1)&\leq 5q^{-4}+5q^{-7}+5q^{-8}.
    \end{align*}
    For $q\geq 2$ we have $3q^{-2}+3q^{-3}\leq \frac 98$, and the remaining values are at most $1$.
\end{proof}

\section{First method: incidence correspondence}\label{sec:incidence}

Recall from Section \ref{sec:bound-reducible} that $t \colonequals t(q,n,d)$ denotes the proportion of degree $d$ hypersurfaces in $\bP^n$ that are reducible over $\F_q$. We use this quantity to compute a lower bound on the proportion of points that satisfy Theorem \ref{thm:main}.

\begin{proposition}\label{prop:less than q}
    Let $n\geq 2$ and $d\geq 2$. If $1\leq r\leq m$, then
    \begin{align*}
    \mu(q,n,d,r)\geq 1-\frac{q^{r-m}+t(d-1)+\Delta(q^r,d)}{q-1}.
    \end{align*}
    If $r>m$, then 
     \begin{align*}
        \mu(q,n,d,r)
        \geq 1-\frac{q^{m-r}\left(1+t(d-1)+\Delta(q^r,d)\right)}{q-1}.
    \end{align*}
\end{proposition}
\begin{proof}
    Consider the incidence correspondence
    $$
    \mathcal{I} \colonequals \{(P, F) \ | \ F(P)=0 \} \subseteq \mathbb{P}^n(\mathbb{F}_{q^r}) \times (\mathcal{S}_{n,d}\setminus\{0\}).
    $$    
    We count the size of $\mathcal{I}$ in two ways. First, we fix each nonzero $F\in\mathcal{S}_{n,d}$ and count the number of points $P\in \mathbb{P}^n(\mathbb{F}_{q^r})$ with $F(P)=0$. Using Lemma~\ref{lem:LW} for the irreducible hypersurfaces and Lemma~\ref{lem:reducible bound} for the rest, we have 
    \begin{align}
       \nonumber\# \mathcal{I} 
       &\leq  (q^m-1) \left( \frac{(1-t)}{q^r-1}\big(q^{rn}-1+q^{rn}\Delta(q^r,d)\big)+ t \left((d-1) q^{r(n-1)} + \frac{q^{rn}-1}{q^r-1}\right)\right)\\
       &=(q^m-1)\left(\frac{q^{rn}-1}{q^r-1}+\frac{(1-t)q^{rn}\Delta(q^r,d)}{q^r-1}+t(d-1)q^{r(n-1)}\right)\nonumber\\
        &\leq \left( \frac{q^{m}-1}{q^r-1} \right) \left((q^{rn}-1)+(1-t)\Delta(q^r,d)q^{rn}+t(d-1)q^{rn}\right).\label{ineq:upper-bound}
    \end{align}    
   
    Next, we fix a point $P\in \mathbb{P}^n(\mathbb{F}_{q^r})$ and count the number of polynomials vanishing at $P$. Let $\mu\colonequals\mu(q,n,d,r)$. We first consider the case $r \leq m$.   Of the points in $\bP^n(\F_{q^r})$, $1-\mu$ of them are in the common vanishing locus of a subspace of $\mathcal{S}_{n,d}$ with dimension at least $m-r+1$, while the remaining $\mu$ are in the vanishing locus of a subspace of dimension $m-r$. Therefore,
    \begin{align}
    \nonumber\# \mathcal{I} &\geq  \left(\frac{q^{r(n+1)}-1}{q^r-1}\right)\left((1-\mu)(q^{m-r+1}-1)+\mu(q^{m-r}-1)\right)\\
    &=\left(\frac{q^{r(n+1)}-1}{q^r-1}\right)
    \left(q^{m-r+1}-1-\mu q^{m-r}(q-1)\right). \label{ineq:lower-bound}
    \end{align}  
    Combining inequalities~\eqref{ineq:upper-bound} and~\eqref{ineq:lower-bound}, and multiplying by $q^r-1$, we obtain
    \begin{align*}
        &(q^{r (n + 1)} - 1) \left(q^{m-r+1}-1 -\mu q^{m-r}(q-1)\right)\\
        &\qquad \leq (q^m-1)(q^{rn}-1)+(q^m-1)q^{rn}\left((1-t)\Delta(q^r,d)+t(d-1)\right),
    \end{align*}
    and after dividing through by $q^{m+rn}$ and rearranging terms, we can conclude that
   \begin{align*}
        (q-1)(1-q^{-r(n+1)})\mu& \geq q^{-m-rn}\left((q^{r(n+1)}-1)(q^{m-r+1}-1)-(q^m-1)(q^{rn}-1)\right.\\
        &\qquad\qquad \quad\left.-(q^m-1)q^{rn}\left((1-t)\Delta(q^r,d)+t(d-1)\right)\right)\\
        &= \left(q-1-q^{r-m}(1-q^{-r})+q^{-rn}(1-q^{1-r})\right)\\
        &\qquad\qquad\quad -(1-q^{-m})\left((1-t)\Delta(q^r,d)+t(d-1)\right)\\
        &\geq q-1-q^{r-m}-(1-t)\Delta(q^r,d)-t(d-1)\\
        &\geq (q-1)-(q^{r-m}+t(d-1)+\Delta(q^r,d)).
    \end{align*}

    Now divide both sides by $q-1$. Since $\mu\geq 0$ and $1-q^{-r(n+1)}\leq 1$, we have
    \[\mu\geq (1-q^{-r(n+1)})\mu\geq 1-\frac{q^{r-m}+t(d-1)+\Delta(q^r,d)}{q-1}\]
    as desired.


    If $r>m$, then we repeat the same argument but replacing inequality~\eqref{ineq:lower-bound} with
    \begin{align*}
    \# \mathcal{I} &\geq  \left(\frac{q^{r(n+1)}-1}{q^r-1}\right)(1-\mu)(q-1),
    \end{align*}
    since a fraction $\mu$ of the points lie on no hypersurface, and the remaining fraction $1-\mu$ lie on at least one. We obtain the desired lower bound on $\mu$ by carrying out a similar algebraic manipulation.
\end{proof}

\begin{remark}\label{rmk: simpler cases}
    If $d=1$, we repeat the proof of Proposition~\ref{prop:less than q} but replace the bound \eqref{ineq:upper-bound} with the exact count of the number of points on a hyperplane:
    \[\#\mathcal{I}=(q^m-1)\left(\frac{q^{rn}-1}{q^r-1}\right).\]
    Following the rest of the argument we obtain the bounds
    \begin{align*}
        \mu(q,n,1,r)&\geq 1-\frac{q^{r-m}-q^{-m}}{q-1}\qquad \text{if }r\leq m,\\
        \mu(q,n,1,r)&\geq 1-\frac{q^{m-r}-q^{-r}}{q-1}\qquad \text{if }r> m.
    \end{align*}
    Since $q^{-|m-r|}\leq 1$ and $q\geq 2$, the lower bound is always positive.
\end{remark}

\subsection{Checking the inequality}\label{subsec: checking inequality}
Recall that the case $d=1$ was handled in Remark~\ref{rmk: simpler cases}. In the following, we assume $d\geq 2$.

To prove $\mu(q,n,d,r)>0$ for $r\leq m$, Proposition~\ref{prop:less than q} reduces the problem to verifying that
\begin{equation}\label{eq:toprove}
    q^{r-m} + t(d-1)+\Delta(q^r,d)
\end{equation}
is less than $q-1$. If instead $r>m$, Proposition~\ref{prop:less than q} reduces the problem to verifying that
\begin{equation}\label{eq:toprove 2}
    q^{m-r}\left(1 + t(d-1)+\Delta(q^r,d)\right)
\end{equation}
is less than $q-1$. Below we show that these bounds hold (and therefore Theorem \ref{thm:main} holds) except possibly in the following cases:
\begin{enumerate}[(i)]
    \item $q\leq 3$, $n=2$, $d\leq 6$, and $r\leq m+1$;
    \item $q=3$ and $r\leq 10$;
    \item $q=2$ and $r\leq 24$;
    \item $q=2$ and $r=m$.
\end{enumerate}
 So given our constraints $n\geq 2$ and $r>\binom{n-1+d}{n-1}$, there are only finitely many quadruples $(q,n,d,r)$ satisfying each of (i)--(iii), and we check these by explicit computation in Appendix~\ref{sec:computer}. Case (iv) is more difficult and will be considered in Sections \ref{sec:avoid} and \ref{sec:inc-exc}.

\medskip

\noindent Since we are assuming $r\geq d+2$, both quantities in \eqref{eq:toprove} and~\eqref{eq:toprove 2} are bounded above by
\begin{align}\label{eq:new toprove}
    1+ t(d-1) + (r-3)(r-4) q^{-r/2} + 5(r-2)^{13/3} q^{-r},
\end{align}
and it suffices to prove that this is less than $q-1$.
Observe that if we have an upper bound for $t(d-1)$ of the form $c$ or $c/q$ for some constant $c>0$, then the quantity~\eqref{eq:new toprove} is bounded by a function $f(q,r)$ of the form \eqref{eq:poly exp form}, so that we can apply Lemma~\ref{lem: poly exp} to obtain the desired bound.

First suppose that $n=2$ and $d\leq 6$. We have $t(d-1)\leq \frac98$ by Lemma \ref{lem:bound t 2}. Applying this bound to \eqref{eq:new toprove}, we obtain an expression that is strictly less than $3$ (and therefore also less than $q-1$) for $q\geq 4$ and all $r$. If instead $q\leq 3$ and $r\geq m+2$, \eqref{eq:toprove 2} is bounded above by 
\[\frac14\left(1+ \frac98 + (r-3)(r-4) q^{-r/2} + 5(r-2)^{13/3} q^{-r}\right).\]
This is less than $1$ for $q=2$ and $r\geq 20$, and is less than $2$ for $q\geq 3$ and all $r$. The cases with $q=2$ and $r\leq 19$ are accounted for in case (iii). Finally, the cases with $q\leq 3$ and $r\leq m+1$ are accounted for in case (i).

In all remaining cases, we have the bound $t(d-1)\leq \frac{1}{2q}$ by Lemma \ref{lem:bound t}. Plugging this bound into \eqref{eq:new toprove}, the resulting expression is less than $3$ for $q\geq 4$ and all $r$, and is less than $2$ for $q=3$ and all $r\geq 11$. If $q=3$ and $r\leq 10$, we have the exceptional case (ii).

Finally, suppose $q=2$. Assume that $r\neq m$, so that $q^{-|m-r|}\leq \frac12$. Then since $r\geq d+2$, \eqref{eq:toprove} and \eqref{eq:toprove 2} are both bounded above by
\[\frac12 + \frac14 + (r-3)(r-4) 2^{-r/2} + 5(r-2)^{13/3} 2^{-r},\]
which is less than $1$ provided that $r\geq 25$.
If $r\leq 24$ or $r=m$, we have the exceptional cases (iii) and (iv), respectively.

\section{Second method: counting by irreducible component}\label{sec:avoid}

In this section, we work on the case $r=m=\binom{n+d}{d}$ more carefully using a different approach. In particular, we present a simple proof for the case $q \geq 3$, which results in a new proof of Theorem~\ref{thm:AGR} for finite fields.

Throughout, assume that $q,n,d$ are fixed and $n \geq 2$. Consider the following collection of hypersurfaces $\mathcal{H}$:
$$
\mathcal{H}=\{H: \text{$H$ is an irreducible hypersurface over $\F_q$ with degree $e\leq d$}\}.
$$

When $r=m$, the quantity $\mu(q,n,d,r)$ measures the proportion of points in $\bP^n(\F_{q^r})$ that do not lie on any degree $d$ hypersurface over $\F_q$. The key inequality in this section is the following. 

\begin{lemma}\label{lem:obs}
$$
\mu(q,n,d,r)\geq 1-\frac{\sum_{H \in \mathcal{H}} \#H(\F_{q^m})}{\# \bP^n(\F_{q^m})}.
$$
\end{lemma}
\begin{proof}
Let $X=\{F=0\}$ be a degree $d$ hypersurface defined over $\F_q$. Factorize $F=F_1F_2\cdots F_k$ into the product of irreducible polynomials over $\F_q$. For each $1 \leq i \leq k$, let $X_i$ be the hypersurface defined by $F_i=0$; then $X_i \in \mathcal{H}$. Thus, $X(\F_{q^m})=\cup_{i=1}^{k}X_i(\F_{q^m}) \subseteq \cup_{H \in \mathcal{H}} H(\F_{q^m})$, and the lemma follows.
\end{proof}

In view of Lemma~\ref{lem:obs}, to prove $\mu(q,n,d,m)>0$, it suffices to show that the sum of $\#H(\F_{q^m})$ over $H\in\mathcal{H}$ is strictly less than the number of points in $\bP^n(\F_{q^m})$. Since each $H\in\mathcal{H}$ is irreducible, an upper bound on $\#H(\F_{q^m})$ follows from Lemma~\ref{lem:LW}. Now we give an upper bound on $\#\mathcal{H}$.

\begin{lemma}\label{lem:H}
    $$
    \#\mathcal{H}\leq \frac{q^m}{q-1} \bigg(1-\frac{1}{q^{mn/(n+d)}}+\frac{2}{q^{m-(n+1)}}\bigg).
    $$
\end{lemma}

\begin{proof}
For each positive integer $j$, let       
\[h_{j}\colonequals\frac{1}{q-1}\left(q^{\binom{n+j}{j}}-1\right)\]
denote the number of hypersurfaces of degree $j$ in $\mathbb{P}^n$ over $\FF_q$. Let $g_{j}$ denote the number of irreducible degree $j$ hypersurfaces in $\mathbb{P}^n$ over $\FF_q$.

For each $j \geq 3$, we claim that
\begin{equation}\label{eq:g_j}
g_{j} \leq h_{j}-2h_{j-1}+h_{j-2}.    
\end{equation}
To prove this, it suffices to give a lower bound on $h_{j}-g_{j}$, namely the number of reducible hypersurfaces over $\F_q$ with degree $j$. We explicitly construct reducible hypersurfaces with degree $j$ of the form $H'\cup L_1$ and $H'\cup L_2$, where $L_1$ and $L_2$ are distinct hyperplanes and $H'$ is any hypersurface of degree $j-1$. We double counted hypersurfaces of the form $H''\cup L_1\cup L_2$, where $H''$ is any hypersurface with degree $j-2$; thus, there are at least $2h_{j-1}-h_{j-2}$ distinct reducible degree $j$ hypersurfaces. Hence, $h_j-g_j\geq 2h_{j-1}-h_{j-2}$, yielding the desired inequality ~\eqref{eq:g_j}. 

It follows from inequality~\eqref{eq:g_j} that
\begin{align}    
\# \mathcal{H}
&= g_{1}+g_{2}+\sum_{j=3}^{d} g_{j} \notag\\
&\leq h_{1}+h_{2}+\sum_{j=3}^{d} (h_{j}-2h_{j-1}+h_{j-2}) \notag\\
&=h_{d}-h_{d-1}+2h_{1} \notag\\
&=\frac{1}{q-1}\bigg(q^{\binom{n+d}{d}}-q^{\binom{n+d-1}{d-1}}+2(q^{n+1}-1)\bigg) \notag\\
&\leq \frac{q^m}{q-1} \bigg(1-\frac{1}{q^{\binom{n+d}{d}-\binom{n+d-1}{d-1}}}+\frac{2}{q^{m-(n+1)}}\bigg).
\end{align}
Finally, we observe that $m=\binom{n+d}{d}=\frac{n+d}{d}\binom{n+d-1}{d-1}$, thus $\binom{n+d}{d}-\binom{n+d-1}{d-1}=\frac{mn}{n+d}$.
\end{proof}

Now we present a simple proof of Theorem~\ref{thm:AGR} for finite fields, which is the main result in \cite{AGR24}.
\begin{proof}[Proof of Theorem~\ref{thm:AGR} for $K$ finite]
Let $K=\F_q$ with $q \geq 3$. In view of \cite[Proposition 3.1]{AGR24}, we can assume $d\geq q \geq 3$. By the reductions in Section \ref{sec:special cases reductions}, we can assume $n\geq 2$. Note that $m=\binom{n+d}{d}$ is greater than both $d$ and $n$.

By Lemma~\ref{lem:LW}, for each $H \in \mathcal{H}$, we have
\begin{align}
\# H(\F_{q^m}) 
&\leq \frac{1}{q^m-1} \big(q^{mn}-1+(d-1)(d-2) q^{m(n-1/2)}+5d^{13/3} q^{m(n-1)}\big) \label{eq:HF2}\\
&\leq \frac{q^{mn}}{q^m-1} \bigg(1+\frac{m^2}{q^{m/2}}+\frac{5m^{13/3}}{q^m}\bigg) \label{eq:HF}.
\end{align}    
This upper bound is uniform across all $H \in \mathcal{H}$; it depends only on $d$ and not on $\deg(H)$. Since $d\geq 3$ and $n \geq 2$, we have 
\begin{equation}\label{eq:m bound}
    m=\binom{n+d}{d}\geq \binom{n+3}{3}=(n+1)\cdot \frac{(n+3)(n+2)}{6}>3(n+1).
\end{equation}
The inequality \eqref{eq:m bound} leads to $m-(n+1) \geq \frac{2m}{3}$. Thus, Lemma~\ref{lem:H} implies that
\begin{align*}
    \#\mathcal{H}&\leq \frac{q^m}{q-1} \bigg(1-\frac{1}{q^{mn/(n+d)}}+\frac{2}{q^{m-(n+1)}}\bigg)\leq \frac{q^m}{q-1} \bigg(1+\frac{2}{q^{2m/3}}\bigg).
\end{align*}
We multiply these bounds on $\# H(\mathbb{F}_{q^m})$ and $\# \mathcal{H}$ to obtain an upper bound on $\sum_{H\in \mathcal{H}}\#H(\F_{q^m})$: 
\begin{equation}\label{eq:upp-bound-on-sum-of-H}
\sum_{H\in \mathcal{H}} \# H(\mathbb{F}_{q^m}) \leq \frac{q^{mn}}{q^m-1} \bigg(1+\frac{m^2}{q^{m/2}}+\frac{5m^{13/3}}{q^m}\bigg) \cdot \frac{q^m}{q-1} \bigg(1+\frac{2}{q^{2m/3}}\bigg).
\end{equation}
Since $q \geq 3$ and $m=\binom{n+d}{d}\geq 10$, we also have a lower bound on $\bP^n(\F_{q^m})$:
\begin{equation}\label{eq:lower-bound-on-Pn(Fqm)}
\#\bP^n(\F_{q^m})= \frac{q^{m(n+1)}}{q^m-1}\left(1-q^{-m(n+1)}\right)\geq \frac{q^{m(n+1)}}{q^m-1}\left(1-3^{-30}\right).
\end{equation}
In light of \eqref{eq:upp-bound-on-sum-of-H} and \eqref{eq:lower-bound-on-Pn(Fqm)}, to show $\sum_{H\in \mathcal{H}}\#H(\F_{q^m})<\#\bP^n(\F_{q^m})$, it suffices to prove
\begin{equation}\label{eq:eq111}
\bigg(1+\frac{m^2}{q^{m/2}}+\frac{5m^{13/3}}{q^m}\bigg) \bigg(1+\frac{2}{q^{2m/3}}\bigg) <(q-1)\left(1-\frac{1}{3^{30}}\right).
\end{equation}
By Lemma~\ref{lem: poly exp}, this inequality holds for all $q\geq 3$ and $m\geq 12$, or for $q\geq 4$ and $m\geq 7$. The only case remaining to verify is $q=3, n=2, d=3$ (so $m=10$); in this case, a point $P \in \bP^n(\F_{q^m})$ with the required property has been explicitly constructed at the end of \cite[Section 6]{AGR24}.
\end{proof}

Next, we consider the case $q=2$. We first prove a result that holds for large $d$.
\begin{theorem}\label{thm:q 2 case}
Let $q=2$ and $r=m$. If $n\geq 2$ and $d\geq\max(6,n+1)$, then Theorem~\ref{thm:main} holds.
\end{theorem}
\begin{proof}
As in inequality~\eqref{eq:HF}, for each $H \in \mathcal{H}$, we have
\begin{align}\label{eq:HF2m}
\# H(\F_{2^m}) 
&\leq \frac{2^{mn}}{2^m-1} \bigg(1+\frac{d^2}{2^{m/2}}+\frac{5d^{13/3}}{2^m}\bigg).
\end{align}
Now we bound $\#\mathcal{H}$ using Lemma~\ref{lem:H}.
Recall from the proof of Lemma~\ref{lem:H} that $\frac{mn}{n+d}=m-\binom{n+d-1}{d-1}$. For $d\geq 6$ we have
$$2^{\binom{n+d-1}{d-1}}\geq 2^{\binom{d+1}{2}}>10d^{13/3},$$
which implies
\begin{equation}\label{eq:half bound 1}
    \frac{1}{2}\cdot\frac{1}{2^{mn/(n+d)}}>\frac{5d^{13/3}}{2^m}.
\end{equation}
Observe that
$$
m\bigg(\frac{1}{2}-\frac{n}{n+d}\bigg) =\frac{n+d}{d}\binom{n+d-1}{d-1}\bigg(\frac{d-n}{2(n+d)}\bigg)=\frac{d-n}{2d}\binom{d+n-1}{n}.$$
For $n\geq 4$, this is greater than or equal to $\frac{1}{2d}\binom{d+3}{4}$ because $d\geq n+1$; for $n=3$ or $n=2$, this is greater than or equal to $\frac{3}{2d}\binom{d+2}{3}$ or $\frac{4}{2d}\binom{d+1}{2}$ respectively because $d\geq 6$. These bounds imply that
$$
2^{m(\frac{1}{2}-\frac{n}{n+d})}>2d^2+6
$$
for all $n\geq 2$ and $d\geq\max(6,n+1)$.
Therefore,
\begin{equation}\label{eq:half bound 2}
    \frac{1}{2}\cdot \frac{1}{2^{mn/(n+d)}}>\frac{d^2+3}{2^{m/2}}.
\end{equation}
Using the bound $m-(n+1)>\frac{m}{2}$, which follows from inequality~\eqref{eq:m bound}, we combine inequalities~\eqref{eq:half bound 1} and~\eqref{eq:half bound 2} with Lemma~\ref{lem:H} to obtain:
$$\#\mathcal{H}< 2^m\left(1-\frac{d^2}{2^{m/2}}-\frac{5d^{13/3}}{2^m}-\frac{1}{2^{m/2}}\right).$$
Multiplying the bound on $\# \mathcal{H}$ by the upper bound on $\# H(\F_{2^m})$ for each $H\in\mathcal{H}$, we conclude
\begin{align*}
    \sum_{H\in\mathcal{H}}\# H(\F_{2^m})&< \frac{2^{m(n+1)}}{2^m-1} \bigg(1+\frac{d^2}{2^{m/2}}+\frac{5d^{13/3}}{2^m}\bigg)\left(1-\frac{d^2}{2^{m/2}}-\frac{5d^{13/3}}{2^m}-\frac{1}{2^{m/2}}\right)\\
    &< \frac{2^{m(n+1)}}{2^m-1} \left(1-\frac{1}{2^{m/2}}\right)< \frac{2^{m(n+1)}-1}{2^m-1},
\end{align*}
proving the desired bound by Lemma~\ref{lem:obs}.
\end{proof}

We also consider the case $q=d=2$ separately in the following result, as the assumption $d=2$ allows us to get much tighter bounds on the number of points on each hypersurface.

\begin{theorem}\label{thm:q=d=2}
Let $q=2$ and $r=m$. Then Theorem \ref{thm:main} holds if $d=2$ and $n\geq 4$.
\end{theorem}
\begin{proof}
As in inequality~\eqref{eq:HF2}, for each $H \in \mathcal{H}$, we have
$$
\# H(\F_{2^m}) 
\leq \frac{2^{mn}}{2^m-1} \left(1+\frac{5 \cdot 2^{13/3}}{2^m}\right).
$$
In this case, we improve the upper bound on $\mathcal{H}$ in Lemma~\ref{lem:H} as follows. Borrowing the notation from Lemma~\ref{lem:H}, note that $g_{2}= h_{2}-\binom{h_{1}}{2}-h_{1}$. It follows that
\begin{align*}
\#\mathcal{H}
&=g_{1}+g_{2}=h_{1}+g_{2}=h_{2}-\binom{h_{1}}{2}=2^{m}-1-\frac{(2^{n+1}-1)(2^{n+1}-2)}{2}\\
&=2^m-1-(2^{n+1}-1)(2^n-1) \leq 2^m-2^{2n}.
\end{align*}
For $n\geq 4$, we have $2^{2n}>5\cdot 2^{13/3}+1$ and so $\#\mathcal{H}< 2^m-5\cdot 2^{13/3}-1$. We conclude that
\begin{align*}
    \sum_{H\in\mathcal{H}}\# H(\F_{2^m})&< \frac{2^{m(n+1)}}{2^m-1} \bigg(1+\frac{5\cdot 2^{13/3}}{2^m}\bigg)\left(1-\frac{5\cdot 2^{13/3}}{2^m}-\frac{1}{2^{m}}\right)\\
    &< \frac{2^{m(n+1)}}{2^m-1} \left(1-\frac{1}{2^{m}}\right)< \frac{2^{m(n+1)}-1}{2^m-1}.\qedhere
\end{align*}
\end{proof}

We remark that the approach above does not allow us to prove the theorem when $d\geq 3$ is small compared to $n$, because the $(d-1)(d-2)q^{-m/2}$ term from the Lang--Weil bounds is significantly larger than the $-q^{-\binom{n+d}{d}+\binom{n+d-1}{d-1}}$ term from the count of irreducible hypersurfaces. Consequently, we need an alternate approach for large $n$.

\section{Third method: inclusion-exclusion}\label{sec:inc-exc}

\subsection{Bounds on reducible intersections of hypersurfaces}\label{sec:singular locus}

Let $H_1,\ldots,H_k$ be randomly chosen degree $d$ hypersurfaces in $\bP^n$ defined over $\F_q$, and $X\colonequals H_1\cap\cdots\cap H_k$. We will show that $X$ is geometrically irreducible of dimension $n-k$ with ``high'' probability.

Counting reducible hypersurfaces ($k=1$) is relatively straightforward, because any reducible hypersurface is a union of hypersurfaces of smaller degree, and there is a natural parametrization of these. However, if $X$ is an intersection of $k\geq 2$ hypersurfaces, the irreducible components of $X$ may no longer each be expressible as an intersection of $k$ hypersurfaces (that is, they may not be complete intersections). As there is no convenient parametrization of the space of $(n-k)$-dimensional subvarieties of $\bP^n$, we will instead use the fact that any reducible variety must have a large singular locus, and bound the number of varieties with large singular locus.

\begin{lemma}\label{lem:geom red singular}
    Let $X/\F_q$ be a complete intersection in $\bP^n$ of dimension $n-k$ with $k\leq n/2$. If $X$ is geometrically reducible, then the singular locus of $X$ has dimension at least $n-2k$.
\end{lemma}
\begin{proof}
    Let $C_1,C_2$ be two distinct irreducible components of $X_{\overline{\F_q}}$. Since $X$ is a complete intersection, $C_1$ and $C_2$ are projective of dimension $n-k$, so their intersection $D$ has dimension at least $n-2k$ and is contained in the singular locus of $X$.
\end{proof}

To bound the number of varieties with large singular locus, we adapt an argument due to Poonen, namely \cite[Lemma 2.6]{P04}. Bucur and Kedlaya also used a version of this technique \cite{BK12}, and we follow their exposition closely as they specifically consider the case of intersections of multiple hypersurfaces. The main difference in our approach is that instead of considering singular points of large degree, we consider singular subvarieties of large dimension; see also Remark~\ref{rem:difference}. A version of this method was previously used by Slavov to bound the number of varieties with large singular locus \cite[Section 6]{S15}.

Let $q\geq 2$ be a prime power and $p$ the characteristic of $\F_q$. We will eventually apply these results to the case $p=q=2$. 

\begin{lemma}\label{lem:vanish on Y}
    Fix $d\geq 1$ and a projective variety $Y\subseteq \bP^n$ of dimension $e\geq 1$.  The proportion of $f\in\mathcal{S}_{n,d}$ vanishing on $Y$ is bounded above by $q^{-\binom{d+e}{e}}$.
\end{lemma}
\begin{proof}
    The Hilbert function $h_Y(d)$ measures the codimension in $\mathcal{S}_{n,d}$ of the space of polynomials vanishing on $Y$, so the desired probability is exactly $q^{-h_Y(d)}$.
    By taking $I$ to be the homogeneous ideal defining $Y$ and applying \cite[Theorem 2.4]{S97}, noting that $\deg I\geq 1$, we conclude that $$h_Y(d)\geq \binom{d+e+1}{e+1}- \binom{d+e}{e+1}=\binom{d+e}{e},$$
    giving the desired bound.
\end{proof}

Given $f\in \mathcal{S}_{n,d}$, let $H_f$ denote the subvariety of $\bP^n$ defined by $f=0$. The next lemma shows that for any irreducible variety $X$ of dimension $b$ with a small singular locus, there are many hypersurfaces $H_f$ for which $X\cap H_f$ has dimension $b-1$ and small singular locus. To accomplish this, we first restrict to a smooth affine open subset $U\subseteq X$ and assume that the first $b$ coordinates of the ambient affine space give local coordinates for $U$.

\begin{lemma}\label{lem: U cap Hf}
    Let $U$ be a smooth $b$-dimensional quasiprojective variety in $\mathbb{P}^n$ over $\FF_q$. Fix integers $d\geq p+1$ and $1\leq c\leq b$. Suppose that $U$ is contained in the standard affine chart with coordinates $t_1,\ldots,t_n$ and that $dt_1,\ldots,dt_b$ freely generate the module $\Omega_{U/\F_q}$ of differential $1$-forms on $U$. Given $f\in\mathcal{S}_{n,d}$ chosen uniformly at random, the probability that $\dim (U\cap H_f)=b$ or $\dim (U\cap H_f)_{\text{sing}}>b-c$ is at most 
    \[q^{-\binom{\lfloor d/p\rfloor +b}{b}}+\deg(\overline{U})\sum_{i=0}^{c-1}(d-1)^iq^{-\binom{\lfloor (d-1)/p\rfloor+b-i}{b-i}}.\]
\end{lemma}
\begin{proof}
     On the affine space with coordinates $t_1,\ldots,t_n$, the elements of $\mathcal{S}_{n,d}$ are given by polynomials of degree at most $d$ in $t_1,\ldots,t_n$. We need to bound the locus of points on $U$ on which $f$ and all derivatives $\frac{\partial f}{\partial t_i}$ simultaneously vanish. Using Poonen's technique \cite[Lemma 2.6]{P04}, we decompose $f$ to decouple the vanishing of $f$ from the vanishing of each derivative. Namely, if we choose $f_0\in\mathcal{S}_{n,d}$, $g_1,\ldots,g_c\in\mathcal{S}_{n,\lfloor (d-1)/p\rfloor}$, and $h\in\mathcal{S}_{n,\lfloor d/p\rfloor}$ each uniformly at random, then 
    \begin{equation}\label{eq: f decomp}
        f\colonequals f_0+g_1^pt_1+\cdots +g_c^pt_c+h^p
    \end{equation}
    will be distributed uniformly in $\mathcal{S}_{n,d}$, whereas the derivative with respect to $t_i$ for $1\leq i\leq c$ depends only on $f_0$ and $g_i$ because we are working over a field of characteristic $p$.
    
    For $i\in \{0,1,\ldots,c\}$, define
    \[W_i\colonequals U\cap\left\{\frac{\partial f}{\partial t_1}=\cdots=\frac{\partial f}{\partial t_i}=0\right\}.\]
    Let $0\leq i\leq c-1$, and suppose that we have already chosen $f_0,g_1,\ldots,g_i$ so as to ensure $\dim(W_i)= b-i$. Let $V_1,\ldots,V_\ell$ be the reduced loci of the $(b-i)$-dimensional irreducible components of $W_i$. Note that $\ell\leq \deg(\overline{U})(d-1)^i$ by B\'ezout's theorem.
    Now select $g_{i+1}\in\mathcal{S}_{n,\lfloor (d-1)/p\rfloor}$ uniformly at random. Fix $1\leq j \leq \ell$. We will bound the probability for which
    $$\frac{\partial f}{\partial t_{i+1}}=\frac{\partial f_0}{\partial t_{i+1}}+g_{i+1}^p$$
    vanishes on $V_j$. If no such $g_{i+1}$ exists then the probability is $0$. Otherwise, let $\gamma$ be a $g_{i+1}$ for which $\frac{\partial f}{\partial t_{i+1}}$ vanishes on $V_j$. Then every $g_{i+1}\in\mathcal{S}_{n,\lfloor (d-1)/p\rfloor}$ can be written $g_{i+1}=\gamma+\varepsilon$ for a uniquely determined $\varepsilon\in \mathcal{S}_{n,\lfloor (d-1)/p\rfloor}$. Now
    $$\frac{\partial f}{\partial t_{i+1}}=\frac{\partial f_0}{\partial t_{i+1}}+\gamma^p+\varepsilon^p,$$
    which equals $\varepsilon^p$ on $V_j$. Since $V_j$ is reduced, $\frac{\partial f}{\partial t_{i+1}}$ vanishes on $V_j$ if and only if $\varepsilon$ vanishes on $V_j$, which by Lemma~\ref{lem:vanish on Y} occurs with probability at most $q^{-\binom{\lfloor (d-1)/p\rfloor+b-i}{b-i}}$. Thus, the proportion of $g_{i+1}$ for which $\partial f/\partial t_{i+1}$ vanishes on at least one component among $V_1,\ldots,V_\ell$ is at most
    $$\ell q^{-\binom{\lfloor (d-1)/p\rfloor+b-i}{b-i}}\leq \deg(\overline{U})(d-1)^i q^{-\binom{\lfloor (d-1)/p\rfloor+b-i}{b-i}}.$$
   Provided that we avoid all these choices of $g_{i+1}$, we have $\dim(W_{i+1})=b-i-1$ and may continue the induction.

    Finally, suppose $f_0,g_1,\ldots,g_c$ have all been chosen in such a way that $\dim(W_c)= b-c$. Now for uniformly selected $h\in\mathcal{S}_{n,\lfloor d/p\rfloor}$, the probability that $f$ vanishes on $U$ is at most $q^{-\binom{\lfloor d/p\rfloor+b}{b}}$ by an argument analogous to the previous paragraph; recall that $U$ is smooth and therefore reduced. So, the probability that $H_f\cap U$ is $b$-dimensional, or that $(H_f\cap U)_{\text{sing}}\subseteq W_c$ has dimension greater than $b-c$, is at most the sum of all the probabilities computed so far, which is 
    \[q^{-\binom{\lfloor d/p\rfloor +b}{b}}+\deg(\overline{U})\sum_{i=0}^{c-1}(d-1)^iq^{-\binom{\lfloor (d-1)/p\rfloor+b-i}{b-i}}.\qedhere\]
\end{proof}

Next, we use Lemma~\ref{lem: U cap Hf} to deduce the following corollary.

\begin{corollary}\label{cor: singular bound}
    Let $1\leq k\leq \frac{n-1}{2}$. Pick $f_1,\ldots, f_k\in\mathcal{S}_{n,d}$ uniformly at random. The probability that $H_{f_1}\cap \cdots\cap H_{f_k}$ has dimension larger than $n-k$, or has singular locus of dimension larger than $n-2k-1$, is bounded above by 
    \begin{equation}\label{eq:big bound}
        (n+1)\sum_{j=0}^{k-1} \binom{n}{j}\left(q^{-\binom{\lfloor d/p\rfloor +n-j}{n-j}}+d^j\sum_{i=0}^{2k-j}(d-1)^iq^{-\binom{\lfloor (d-1)/p\rfloor+n-j-i}{n-j-i}}\right).
    \end{equation}
\end{corollary}
\begin{proof}    
    For $n,k$ as above and $0\leq j\leq k$, say a variety is ``$j$-bad'' if it has dimension larger than $n-j$ or if it has singular locus of dimension larger than $n-2k-1$. Let $0\leq j\leq k-1$, and suppose that we are given $X_j=H_{f_1}\cap \cdots \cap H_{f_j}$ (or $X_0=\bP^n$ if $j=0$) which is \emph{not} $j$-bad. Note that such $X_j$ has dimension equal to $n-j$, is geometrically irreducible by Lemma~\ref{lem:geom red singular}, and has degree $d^j$. We will now choose $f_{j+1}\in\mathcal{S}_{n,d}$ uniformly at random. Since $X_j$ is not $j$-bad, its singular locus has dimension at most $n-2k-1$. Suppose that $X_j\cap H_{f_{j+1}}$ is $(j+1)$-bad. Then the same is true upon removing a subvariety of dimension at most $n-2k-1$, so $(X_j)_{\text{smooth}}\cap H_{f_{j+1}}$ is also $(j+1)$-bad; furthermore, if we have an open cover of $(X_j)_{\text{smooth}}$, then there exists some $U$ in the open cover for which $U\cap H_{f_{j+1}}$ is $(j+1)$-bad. This means that we can compute the probability that $U\cap H_{f_{j+1}}$ is $(j+1)$-bad for each $U$ separately, and add the results together to obtain an upper bound on the probability that $X_j\cap H_{f_{j+1}}$ is $(j+1)$-bad.

    Pick a standard affine open $\bA^n$ in $\bP^n$ and let $Y$ be the restriction of $(X_j)_{\text{smooth}}$ to $\bA^n$. Now choose a subset $S$ of $\{1,\ldots,n\}$ with $b:=n-j$ elements, and let $U_S\subseteq Y$ be the open subvariety on which $dt_i$ for $i\in S$ freely generate $\Omega_{U_S/\F_q}$. Applying Lemma~\ref{lem: U cap Hf} with $b=n-j$ and $c=2k+1-j$ (note that $c\leq b$ since $k\leq \frac{n-1}{2}$, and $b-c=n-2k-1$), we find that the probability that $U_S\cap H_{f_{j+1}}$ is $(j+1)$-bad is bounded above by 
    \[q^{-\binom{\lfloor d/p\rfloor +n-j}{n-j}}+d^j\sum_{i=0}^{2k-j}(d-1)^iq^{-\binom{\lfloor (d-1)/p\rfloor+n-j-i}{n-j-i}}.\]
    Since the sets $U_S$ cover $Y$, and there are $n+1$ choices for standard affine open, we can multiply this by $(n+1)\binom{n}{j}$ for an upper bound on the probability that $X_{j+1}$ is $(j+1)$-bad. If we avoid this event, we have a variety $X_{j+1}\colonequals X_j\cap H_{f_{j+1}}$ which is not $(j+1)$-bad and can continue the induction.

    In conclusion, the probability that $H_{f_1}\cap \cdots\cap H_{f_k}$ is $k$-bad can be bounded above by adding the probabilities we obtained for each $j=0,\ldots,k-1$, resulting in the upper bound from the statement of the corollary.
\end{proof}

\begin{remark}\label{rem:difference}
    Lemma~\ref{lem: U cap Hf} and Corollary~\ref{cor: singular bound} closely parallel \cite[Lemma 2.6]{BK12} and \cite[Corollary 2.7]{BK12}, respectively. In contrast, Bucur and Kedlaya \cite{BK12} obtained bounds in a considerably simplified form. These weaker bounds are sufficient for their purposes, as they were primarily interested in the asymptotics for large $d$. If we were to relax the bounds in a similar fashion, then we would obtain considerably worse bounds in Corollary \ref{cor:1 over d3} below, and checking the remaining exceptional cases would be computationally infeasible. 
\end{remark}

Now we specialize to the case $q=p=2$. We use the geometric computations above to obtain the desired bound on the probability that an intersection of hypersurfaces fails to be geometrically irreducible of the expected dimension.

\begin{corollary}\label{cor:1 over d3}
    Assume that $d\geq 3$ and $n\geq 36$, or $d\geq 5$ and $n\geq 13$, or $d\geq 7$ and $n\geq 11$, or $d\geq 9$ and $n\geq 10$. Let $k \in \{1,2,3\}$. If $X/\F_2$ is an intersection of $k$ elements of $\mathcal{S}_{n,d}$ selected uniformly at random, then the probability that $X$ fails to be a geometrically irreducible variety of dimension $n-k$ is less than $\frac{6}{5d^3}$.
\end{corollary}
\begin{proof}
Note that $k\leq \frac{n-1}{2}$ under any of the assumptions listed in the statement. By Lemma~\ref{lem:geom red singular}, if $X$ fails to be a geometrically irreducible variety of dimension $n-k$, then either $X$ has dimension larger than $n-k$, or $X$ has singular locus of dimension larger than $n-2k-1$. Thus, we can apply Corollary~\ref{cor: singular bound} to deduce that the desired probability is bounded above by the expression~\eqref{eq:big bound}.

    If we multiply the expression~\eqref{eq:big bound} by $d^3$, it suffices to show that the resulting expression is bounded above by $\frac 65$. Write $d=2u$ if $d$ is even, and $d=2u-1$ if $d$ is odd. In either case, the resulting expression is bounded above by
    \begin{align*}
        (2u)^3(n+1)\sum_{j=0}^{k-1} \binom{n}{j}\left(2^{-\binom{u-1+n-j}{n-j}}+(2u)^{j}\sum_{i=0}^{2k-j}(2u-1)^i2^{-\binom{u-1+n-j-i}{n-j-i}}\right).
    \end{align*}
    For fixed $u\in \{2,3,4\}$ and $k\in\{1,2,3\}$, this expression is a sum where each term is a polynomial in $n$ times $2$ to the power of a polynomial in $n$. We can use Lemma~\ref{lem: poly exp} with the variable $t=n$ to determine that this expression is smaller than $\frac{6}{5}$ if $u=2$ and $n\geq 36$, or if $u=3$ and $n\geq 13$, or if $u=4$ and $n\geq 11$. This establishes the desired result for $d\in \{3, 4, 5, 6, 7, 8\}$.  

    To handle the remaining case $d\geq 9$ (or equivalently, $u\geq 5$), we obtain a slightly weaker upper bound by replacing $2u-1$ in the expression above with $2u$, that is,
    \begin{align}\label{eq:weakened bound}
        (n+1)\sum_{j=0}^{k-1} \binom{n}{j}\left((2u)^{3}2^{-\binom{u-1+n-j}{n-j}}+\sum_{i=0}^{2k-j}(2u)^{i+j+3}2^{-\binom{u-1+n-j-i}{n-j-i}}\right).
    \end{align}
    We will prove that the expression~\eqref{eq:weakened bound} is less than $\frac 65$ for all $u\geq 5$ and $n\geq 10$. We can check this for $u=5$ and $n\geq 10$ as above using Lemma~\ref{lem: poly exp} with the variable $t=n$. 
    
    Now fixing any $n\geq 10$, we will apply Lemma~\ref{lem: poly exp} again, this time with variable $t=u$ and with $t_0=z=5$. The first condition of Lemma~\ref{lem: poly exp} is easy to check and the third condition follows from the previous paragraph. Checking the second condition requires more work because of the dependence on $n$. To this end, let $0\leq \ell'\leq\ell \leq 6$ and set $f(u)=(2u)^{\ell'+3}$ and $g(u)=\binom{u-1+n-\ell}{n-\ell}$. Up to constant multiples depending on $n$, every term in (\ref{eq:weakened bound}) has the form $C f(u)2^{-g(u)}$ for some $\ell',\ell$ and some positive constant $C$. 
    We have 
    \[\frac{f'(u)}{f(u)} = \frac{\ell'+3}{u}<2\]
    for $u\geq 5$. On the other hand, $g(u)$ is convex for $u\geq 4$, so 
    \begin{align*}
        g'(u)\geq g(u)-g(u-1)=\binom{u-1+n-\ell-1}{n-\ell-1},
    \end{align*}
    which is greater than $4$ for $u\geq 5$. Thus $f(u)g'(u)\log 2>2f(u)>f'(u)$ for all $u\geq 5$, verifying the second condition.
\end{proof}

\subsection{Inclusion-exclusion}

We continue to assume $q=2$ and $r=m$, so we must bound the number of points on the union of all degree $d$ hypersurfaces over $\F_2$. For any fixed $n$, we can prove the theorem for sufficiently large $d$ by Theorem \ref{thm:q 2 case}, and for all remaining values of $d$ by finite computation. So, without loss of generality, we can take $n$ to be large; in particular, from now on assume $n\geq 10$.

Our main inequality in this section is a variant of the inclusion-exclusion principle. Given a subspace $L\subseteq \mathcal{S}_{n,d}$, let $X_L\colonequals\bigcap_{f\in L}H_f$. We define the following collections of subspaces of $\mathcal{S}_{n,d}$.
\begin{align*}
    \mathcal{L}_1&\colonequals\left\{L\subseteq \mathcal{S}_{n,d} \ | \ \dim_{\F_q}L=1\right\},\\
    \mathcal{L}_2&\colonequals\left\{L\subseteq \mathcal{S}_{n,d} \ | \ \dim_{\F_q}L=2, \ \dim X_L=n-2, \ X_L \text{ geometrically irreducible}\right\},\\
    \mathcal{L}_3&\colonequals\left\{L\subseteq \mathcal{S}_{n,d} \ | \ \dim_{\F_q}L=3, \ \text{there exists }L'\in\mathcal{L}_2 \text{ with }L'\subseteq L\right\}.
\end{align*}
\begin{lemma}\label{lem: inclusion exclusion}
    \begin{align*}
    \#\left(\bigcup_{f\in\mathcal{S}_{n,d}\setminus \{0\}} H_f(\FF_{q^r})\right)&\leq\sum_{L\in\mathcal{L}_1} \#X_L(\FF_{q^r})-q\sum_{L\in\mathcal{L}_2} \#X_L(\FF_{q^r})\\
    &\qquad +(q^3+q^2+q)\sum_{L\in\mathcal{L}_3} \#X_L(\FF_{q^r}).
\end{align*}
\end{lemma}
\begin{proof}
    For each $P\in\bP^n(\FF_{q^r})$ that lies in some $H_f(\F_{q^r})$, we will show that it contributes at least $1$ to the right-hand side of the inequality.
    Let $\mathcal{S}_P$ denote the linear system of polynomials vanishing at $P$, so that $P$ is in $X_L(\F_{q^r})$ if and only if $L\subseteq \mathcal{S}_P$. 
    If $\mathcal{S}_P$ does not contain any subspace $L\in\mathcal{L}_2$, then $P$ is counted at least once by the sum over $\mathcal{L}_1$ and not at all by the remaining two sums. If $\mathcal{S}_P$ is an element of $\mathcal{L}_2$, then it has $q+1$ one-dimensional subspaces, so the contribution to the right-hand side from $P$ is $(q+1)-q(1)=1$.

    Now suppose $\mathcal{S}_P$ is $k$-dimensional for $k\geq 3$. Let $b_2$ and $b_3$ denote the number of subspaces of $\mathcal{S}_P$ in $\mathcal{L}_2$ and $\mathcal{L}_3$ respectively. 
    We count the number of flags $L_2\subseteq L_3\subseteq \mathcal{S}_P$ with $L_2\in\mathcal{L}_2$ and $L_3\in\mathcal{L}_3$ in two ways. On one hand, each $L\in\mathcal{L}_2$ is contained in exactly $\frac{q^{k-2}-1}{q-1}$ $3$-dimensional subspaces, and these are all in $\mathcal{L}_3$ by definition. On the other hand, each $L\in\mathcal{L}_3$ contains at most $\frac{q^3-1}{q-1}$ subspaces in $\mathcal{L}_2$. Therefore
    \[\frac{q^{k-2}-1}{q-1}b_2\leq \frac{q^3-1}{q-1}b_3.\]
    The total contribution of $P$ to the right-hand side is therefore
    \begin{align*}
        &\frac{q^k-1}{q-1}-qb_2+q\frac{q^3-1}{q-1}b_3\geq \frac{q^k-1}{q-1}+\left(\frac{q^{k-2}-1}{q-1}-1\right)qb_2\geq 1
    \end{align*}
    because $b_2\geq 0$ and $k\geq 3$.
\end{proof}

Now we are ready to prove the main result of the section.

\begin{theorem}\label{thm:inc-exc}
    Let $q=2$ and $r=m$. Then Theorem \ref{thm:main} holds in any of the following cases:
    \begin{itemize}
        \item $d\geq 3$ and $n\geq 36$;
        \item $d\geq 5$ and $n\geq 13$;
        \item $d\geq 7$ and $n\geq 11$;
        \item $d\geq 9$ and $n\geq 10$.
    \end{itemize}
\end{theorem}
\begin{proof}
    The conditions on $n$ and $d$ are the same as those in Corollary~\ref{cor:1 over d3}. We will use Corollary~\ref{cor:1 over d3} to produce an upper bound on the right-hand side of    
    Lemma \ref{lem: inclusion exclusion}, and show that the upper bound is less than the total number of points in $\bP^n(\F_{q^m})$.
    
   For the first sum involving $\mathcal{L}_1$, we provide an upper bound by following the same strategy as in the proof of Proposition~\ref{prop:less than q}. The proportion $t$ of $f\in\mathcal{S}_{n,d}\setminus\{0\}$ defining reducible hypersurfaces is bounded by the proportion of $f\in\mathcal{S}_{n,d}$ for which $H_f$ is geometrically reducible or has dimension $n$; hence, Corollary~\ref{cor:1 over d3} with $k=1$ implies that $t\leq \frac{6}{5d^3}$. Using Lemma~\ref{lem:LW} for the 
   irreducible hypersurfaces and Lemma~\ref{lem:reducible bound} for the rest, we obtain the following upper bound on $\sum_{L\in\mathcal{L}_1} \#X_L(\F_{2^m})$:  \begin{align*}
     &(2^m-1) (1-t) \left[\frac{2^{mn}-1+2^{mn} \Delta(2^m, d)}{2^{m}-1} \right] + (2^m-1)t\left[\frac{2^{mn}-1}{2^m-1} + (d-1) 2^{m(n-1)} \right]\\
     &=(2^{mn}-1) + (1-t)2^{mn}\Delta(2^{m}, d) + (2^{m}-1) t (d-1) 2^{m(n-1)},
    \end{align*}
    where we canceled the two terms involving $t(2^{mn}-1)$. Since $t\leq \frac{6}{5d^3}$, we deduce that
     \begin{align}\label{eq:L1bound}
        \sum_{L\in\mathcal{L}_1} \#X_L(\F_{2^m})
        &\leq \frac{2^{m(n+1)}}{2^m-1}\left(1-2^{-m}+\frac{6(d-1)}{5d^3}+\Delta(2^m,d)\right).
    \end{align}

    For the second sum involving $\mathcal{L}_2$, we bound it from below. By Corollary~\ref{cor:1 over d3}, there are at most $\frac{6}{5d^3}2^{2m}$ pairs $(f_1,f_2)\in \mathcal{S}_{n,d}^2$ for which $H_{f_1}\cap H_{f_2}$ is geometrically reducible or has dimension greater than $n-2$. Each of the remaining pairs form a basis of some linear system in $\mathcal{L}_2$. Since each $2$-dimensional vector space over $\F_2$ has $6$ bases, we have
    \[\#\mathcal{L}_2\geq \frac{2^{2m}}6\left(1-\frac{6}{5d^3}\right).\]
    Next, we need a lower bound on the size of $X_L(\F_{2^m})$ for each $L\in\mathcal{L}_2$. Note that $X_L$ has dimension $n-2$ and degree $d^2$. Since  $2^m>2(n-1)d^4= 2(\dim X_L+1) (\deg X_L)^2$ holds for all $n,d\geq 2$, we apply Lemma~\ref{lem:LW not hypersurface} to the variety $X_L$ to obtain
    \begin{align}\label{eq:L2bound}
        \sum_{L\in\mathcal{L}_2} \#X_L(\F_{2^m})&\geq \frac{2^{m(n+1)}}{6(2^m-1)}\left(1-\frac{6}{5d^3}\right)\left(1-2^{-m(n-1)}-\Delta(2^m,d^2)\right).
    \end{align}
    
    For the third sum involving $\mathcal{L}_3$, we bound it from above. The size of $\mathcal{L}_3$ is bounded above by the total number of $3$-dimensional linear systems of hypersurfaces, which is
    \[\frac{(2^m-1)(2^m-2)(2^m-2^2)}{(2^3-1)(2^3-2)(2^3-2^2)}\leq \frac{2^{3m}}{168}.\] 
    For $L\in\mathcal{L}_3$, let $L'\in\mathcal{L}_2$ with $L'\subseteq L$. Then $X_{L'}$ is geometrically irreducible of dimension $n-2$, and $X_L$ is the intersection of this variety with a hypersurface that does not contain it; therefore $X_L$ has dimension $n-3$. Next, we bound the number of elements of $\mathcal{L}_3$ that define geometrically reducible varieties. By Corollary~\ref{cor:1 over d3}, there are at most $\frac{6}{5d^3}2^{3m}$ triples $(f_1,f_2,f_3)\in \mathcal{S}_{n,d}^3$ for which $H_{f_1}\cap H_{f_2}\cap H_{f_3}$ is geometrically reducible. The same upper bound holds after we exclude all the linearly dependent triples. Dividing by the number of bases for a $3$-dimensional space over $\F_2$, we have at most $\frac{1}{168}\cdot\frac{6}{5d^3}2^{3m}$ three-dimensional spaces $L$ for which $X_L$ is geometrically reducible.

    Note that $X_L$ has degree $d^3$. Since  $2^m>2(n-2)d^6= 2(\dim X_L+1) (\deg X_L)^2$ holds for all $n,d\geq 2$, 
    we can use Lemma~\ref{lem:LW not hypersurface} to bound the contributions from geometrically irreducible varieties and Lemma \ref{lem:reducible bound} to bound the contributions from those that are geometrically reducible:
    \begin{align}
        \sum_{L\in\mathcal{L}_3} \# X_L(\FF_{2^m}) &\leq \frac{2^{3m}}{168}\left(1-\frac{6}{5d^3}\right)\left(\frac{2^{m(n-2)}+2^{m(n-2)}\Delta(2^m,d^3)}{2^m-1}\right) \notag \\
        &\qquad +   \frac{2^{3m}}{168}\cdot\frac{6}{5d^3}\left(\frac{2^{m(n-2)}+(d^3-1)(2^{m(n-2)}+1)}{2^m-1}\right) \notag \\
        &\leq \frac{2^{m(n+1)}}{168(2^m-1)}\left(1+\frac{6(d^3-1)}{5d^3}(1+2^{-m(n-2)})+\Delta(2^m,d^3)\right).\label{eq:L3bound}
    \end{align}
    We now substitute the inequalities \eqref{eq:L1bound}, \eqref{eq:L2bound}, and \eqref{eq:L3bound} into Lemma~\ref{lem: inclusion exclusion}. Using the assumption $d\geq 3$, we obtain an upper bound for the number of $\F_{2^m}$-points on all hypersurfaces of degree $d$ defined over $\F_2$:
    \begin{align*}
    &\frac{2^{m(n+1)}}{2^m-1}\biggl(\left(1+\frac{6(d-1)}{5d^3}\right)-\frac26\left(1-\frac6{5d^3}\right)+\frac{14}{168}\left(1+\frac{6(d^3-1)}{5d^3}\right) \\
    &\qquad\qquad +\Delta(2^m,d)+\frac{2}{6}\left(1-\frac{6}{5d^3}\right)\Delta(2^m,d^2)+\frac{14}{168}\Delta(2^m,d^3)\\
    &\qquad\qquad -2^{-m}+\frac{2}{6}\left(1-\frac{6}{5d^3}\right)2^{-m(n-1)}+\frac{14}{168}\cdot \frac{6(d^3-1)}{5d^3}2^{-m(n-2)}\biggr)\\
    &\qquad\leq \frac{2^{m(n+1)}}{2^m-1}\left(\left(\frac{49}{45}-\frac13\cdot \frac{43}{45}+\frac{1}{12}\cdot\frac{97}{45}\right)+\frac{17}{12}\Delta(2^m,d^3)+2^{-m}\right) \\
    &\qquad\leq \frac{2^{m(n+1)}}{2^m-1}\left(\frac{19}{20}+\frac{17}{12}\Delta(2^m,m^3)+2^{-m}\right),
    \end{align*}
where in the last step we used $d\leq m$. 
By Lemma \ref{lem: poly exp} applied with the variable $t=m$,
we have
\[\frac{19}{20}+\frac{17}{12}\Delta(2^m,m^3)+2^{-m}<1-2^{-10}\]
as long as $m\geq 93$, which is guaranteed by our assumptions since $d\geq 3$ and $n\geq 10$. Since $m(n+1)>10$, we conclude that the total number of $\F_{2^m}$-points on hypersurfaces of degree $d$ defined over $\F_2$ is strictly less than $\frac{2^{m(n+1)}-1}{2^m-1}$. So provided the conditions of Corollary~\ref{cor:1 over d3} hold, Theorem~\ref{thm:main} also holds.
\end{proof}

At this point, we have completed the proof of Theorem~\ref{thm:main} for all but finitely many cases. The remaining cases can be computationally checked; see Appendix~\ref{sec:computer} for details.   

\section{Applications to linear families of hypersurfaces}\label{sec:linear-systems}

We recall Question \ref{quest:linear-systems}: given a property $\mathcal{P}$ of an algebraic hypersurface, what is the maximum (projective) dimension of a linear system $\mathcal{L}$ of hypersurfaces in $\mathbb{P}^n$ with degree $d$ such that every $\mathbb{F}_q$-member of $\mathcal{L}$ satisfies property $\mathcal{P}$? 

One condition $\mathcal{P}$ we may consider is the following: for any fixed $2\leq i\leq d$, we require that every $\mathbb{F}_q$-member of $\mathcal{L}$ has an $\mathbb{F}_q$-irreducible factor of degree $i$ or larger. We grant ourselves even more flexibility by introducing a condition that allows this condition to ``barely'' fail, by permitting specific irreducible factors of degree $i-1$.

Given a vector space $V$ over $\F_q$ we use $\bP(V)$ to denote its projectivization. For the rest of the section, we will consider linear systems of hypersurfaces as subsets of $\bP(\mathcal{S}_{n,d}(\mathbb{F}_q))$. For $G_1,\ldots,G_j\in \bP(\mathcal{S}_{n,d}(\mathbb{F}_q))$, we use the notation $\langle G_1,\ldots,G_j\rangle$ for the subspace of $\bP(\mathcal{S}_{n,d}(\mathbb{F}_q))$ spanned by $G_1,\ldots,G_j$.

\begin{definition}\label{def:prescribed-factors} Let $2\leq i\leq  d$ and $0\leq j\leq \binom{n+i-1}{n}-1$. A linear system $\mathcal{L}\subseteq \bP(\mathcal{S}_{n,d}(\mathbb{F}_q))$ has \emph{property $\mathcal{P}_{i, j}$} if there exist polynomials $G_1, G_2, \dots, G_{j}$, each with degree $i-1$, such that every $F\in \mathcal{L}$ satisfies:
\begin{enumerate}
\item\label{cond:prescribed-degree} $F$ has an $\mathbb{F}_q$-irreducible factor of degree at least $i$, or
\item\label{cond:prescribed-factor} $F$ has an $\mathbb{F}_q$-irreducible factor $G$ of degree $i-1$, where $G\in \langle G_1, \dots, G_{j} \rangle$.
\end{enumerate}
\end{definition}

The next result determines the maximum dimension of an $\mathbb{F}_q$-linear system that satisfies the property $\mathcal{P}_{i, j}$ for certain ranges of $i$ and $j$. 

\begin{theorem}\label{thm:system-prescribed-factors} Let $d\geq 2$ and suppose $2\leq i\leq  d$ and $0\leq j\leq n$. There exists an $\mathbb{F}_q$-linear system $\mathcal{L}\subseteq \bP(\mathcal{S}_{n,d}(\mathbb{F}_q))$ with (projective) dimension $\binom{n+d}{n}-\binom{n+i-1}{n}+(j-1)$ that satisfies the property $\mathcal{P}_{i, j}$. Moreover, the result is sharp:  $\dim(\mathcal{L})$ cannot be increased to $\binom{n+d}{n}-\binom{n+i-1}{n}+j$.
\end{theorem}

Note that Theorem~\ref{thm:system-prescribed-degree} follows immediately by setting $j=0$. Indeed, condition~\eqref{cond:prescribed-factor} in Definition~\ref{def:prescribed-factors} is vacuous when $j=0$; that is, the property $\mathcal{P}_{i, 0}$ precisely stands for ``having an irreducible factor of degree at least $i$" in the framework of Question~\ref{quest:linear-systems}. Furthermore, applying Theorem~\ref{thm:system-prescribed-degree} with $i=d$ recovers \cite[Theorem 1.3]{AGR24} for all finite fields $\mathbb{F}_q$, including the case $q=2$. Note that \cite[Theorem 1.3]{AGR24} was stated with the hypothesis $q>2$ due to its dependence on Theorem~\ref{thm:AGR}.

\begin{proof} 
We will first prove the existence, then the sharpness. 

\medskip

\noindent\textbf{Existence.} By applying Theorem~\ref{thm:main} for hypersurfaces of degree $i-1$ with $r=\binom{n+i-1}{n}-j$, there exists a point $P\in\bP^n(\mathbb{F}_{q^r})$ such that the $\mathbb{F}_q$-vector space of all degree $i-1$ hypersurfaces passing through $P$ (and its Galois orbit) has dimension $j$. Let $G_{1}, G_{2}, \ldots, G_{j}$ be an $\mathbb{F}_q$-basis for this space. Consider the linear system $\mathcal{L}_0$ consisting of all hypersurfaces of degree $d$ passing through $P$ and its Galois orbit. Then
$$
\dim(\mathcal{L}_0) \geq \binom{n+d}{n} -\left(\binom{n+i-1}{n}-j\right)-1.
$$
Suppose $F$ is an $\mathbb{F}_q$-member of $\mathcal{L}_0$. Let $G$ be an $\mathbb{F}_q$-irreducible factor of $F$ with $G(P)=0$. If $\deg(G)\geq i$, then $F$ satisfies condition \eqref{cond:prescribed-degree}. Otherwise, $\deg(G)\leq i-1$. We claim that $\deg(G)=i-1$. If $\deg(G)\leq i-2$, then with $k=i-1-\deg(G)$, the polynomials
\[
x_0^k G, \ x_1^k G,\ \ldots,\ x_n^k G
\]
are $n+1$ linearly independent polynomials of degree $i-1$ vanishing at $P$; this contradicts the definition of $P$ and the hypothesis that $j\leq n$. Since $\deg(G)=i-1$, it follows that $G\in \langle G_1, G_2, ..., G_{j}\rangle$ and thus $F$ satisfies condition \eqref{cond:prescribed-factor}. Thus, we have produced a linear system $\mathcal{L}_0$ with (projective) dimension $\binom{n+d}{n}-\binom{n+i-1}{n}+(j-1)$ that satisfies $\mathcal{P}_{i, j}$. 

\medskip

\noindent\textbf{Sharpness.} 
Suppose $\mathcal{L}$ is a linear system with dimension at least $\binom{n+d}{n}-\binom{n+i-1}{n}+j$. We aim to show that $\mathcal{L}$ does \emph{not} satisfy the property $\mathcal{P}_{i, j}$. To this end, let $G_1,\ldots,G_j$ be an arbitrary collection of polynomials of degree $i-1$. Without loss of generality, assume $x_0$ is not in $\langle G_1,\ldots,G_{j}\rangle$; this only matters if $i=2$, in which case we can achieve this by re-indexing coordinates since $j\leq n<n+1$.

Let $\mathcal{A}\subseteq \mathbb{P}(\mathcal{S}_{n, d}(\mathbb{F}_q))$ be a codimension $j$ linear space defined over $\mathbb{F}_q$ that is disjoint from $$\mathbb{P}\langle x_0^{d-i+1} G_1, \ldots, x_{0}^{d-i+1} G_{j}\rangle \cong \mathbb{P}^{j-1}.$$ Consider the linear space $\mathcal{R}_{i, j}\subseteq \mathbb{P}(\mathcal{S}_{n, d}(\mathbb{F}_q))$ defined as the intersection 
$$
\mathbb{P}(\{x_0^{d-i+1} T \ | \ \deg(T)=i-1 \})\cap \mathcal{A}.
$$
 The (projective) dimension of $\mathcal{R}_{i, j}$ satisfies the lower bound
$$
\dim(\mathcal{R}_{i, j}) \geq  \binom{n+i-1}{n}-j-1.
$$
Since $\dim(\mathcal{L})+\dim(\mathcal{R}_{i, j})\geq \binom{n+d}{n}-1$, the two spaces meet in the parameter space $\bP(\mathcal{S}_{n,d}(\F_q))$
of degree $d$ hypersurfaces. Let $E\in \mathcal{L}\cap \mathcal{R}_{i, j}$. Then, $E=x_0^{d-i+1} T$ for some $T$ with $\deg(T)=i-1$, so $E$ does not satisfy condition~\eqref{cond:prescribed-degree}. 

We show that $E$ does not satisfy condition~\eqref{cond:prescribed-factor} either. Assume, to the contrary, that $E$ has an $\mathbb{F}_q$-irreducible factor $G$ belonging to the linear system $\langle G_1, \dots, G_{j}\rangle$. Then we can write $E = G\cdot H$, where $\deg(G)=i-1$ and $\deg(H)=d-i+1$. Since $G$ is irreducible over $\mathbb{F}_q$ and $G\neq \lambda\cdot x_0$ for any $\lambda\in\mathbb{F}_q$, it follows that $\gcd(x_0^{d-i+1}, G)=1$. Combining this with the equality $x_0^{d-i+1} T =E= GH$, we obtain $H=c x_0^{d-i+1}$ for some scalar $c \in \F_q$. Then $E=c x_0^{d-i+1} G\in \mathcal{R}_{i, j}\subseteq \mathcal{A}$, contradicting the definition of $\mathcal{A}$. Hence, $E$ does not satisfy condition~\eqref{cond:prescribed-factor}. We conclude that $\mathcal{L}$ does not have the property $\mathcal{P}_{i, j}$. 
\end{proof}

Next, we address Question~\ref{quest:linear-systems} when the property $\mathcal{P}$ denotes ``is reduced." While reducedness has a standard scheme-theoretic meaning, we also have a more elementary
definition in the case of hypersurfaces. Recall that a homogeneous polynomial $F\in \mathbb{F}_q[x_0, \ldots, x_n]$ is called \emph{squarefree} if, in the (unique) factorization $F=F_1 F_2 \cdots F_{\ell}$ into $\mathbb{F}_q$-irreducible factors, no $F_i$ is repeated. A hypersurface $X=\{F=0\}$ is called \emph{reduced} if $F$ is squarefree.

\begin{corollary}\label{cor:linear-systems-reduced} Let $d\geq 2$. There exists an $\mathbb{F}_q$-linear system $\mathcal{L}\subseteq \bP(\mathcal{S}_{n,d}(\mathbb{F}_q))$ with (projective) dimension $\binom{n+d}{n}-\binom{n+d-2}{n}-1$ where every $\mathbb{F}_q$-member of $\mathcal{L}$ is a reduced hypersurface of degree $d$. Moreover, the result is sharp: $\dim(\mathcal{L})$ cannot be increased to $\binom{n+d}{n}-\binom{n+d-2}{n}$.
\end{corollary}

\begin{proof} \textbf{Existence.} Using $i=d-1$ in Theorem~\ref{thm:system-prescribed-degree}, there exists a linear system $\mathcal{L}$ of degree $d$ hypersurfaces with $\dim(\mathcal{L})=\binom{n+d}{n}-\binom{n+d-2}{n}-1$ where each $\mathbb{F}_q$-member $X=\{F=0\}$ has an irreducible factor of degree at least $d-1$; in particular, $F$ is squarefree, and hence $X$ is reduced.

\textbf{Sharpness.} Let $\mathcal{L}$ be a linear system of degree $d$ hypersurfaces with $\dim(\mathcal{L})=\binom{n+d}{n}-\binom{n+d-2}{n}$. Consider the linear space $\mathcal{R}_{d-1, 0} = \mathbb{P}(\{x_0^2 T \ | \deg(T) = d-2\})$ from the proof of Theorem~\ref{thm:system-prescribed-factors}. Then $\mathcal{L}\cap\mathcal{R}_{d-1, 0}$ has a nontrivial intersection in $\mathbb{P}(\mathcal{S}_{n, d}(\mathbb{F}_q))$, yielding a non-reduced $\mathbb{F}_q$-member of $\mathcal{L}$. \end{proof}

\begin{remark}
By slightly modifying the above proof, Corollary~\ref{cor:linear-systems-reduced} can be generalized by replacing the condition that $F$ is squarefree with cubefree, or more generally $k$-free for any $k\leq d-1$. In this general case, the maximum attainable projective dimension of a linear system where every $\F_q$-member is $k$-free
is $\binom{n+d}{n}-\binom{n+d-k}{n}-1$.
\end{remark}

\begin{remark} Let $Y=\{Q=0\}$ be a fixed hypersurface of degree $d-i+1$. We define a property $\mathcal{P}_{Y, j}$ analogous to Definition~\ref{def:prescribed-factors} as follows. A linear system $\mathcal{L}$ of hypersurfaces is said to have \emph{property $\mathcal{P}_{Y, j}$} if there exist polynomials $G_1, ..., G_{j}$ of degree $i-1$ such that for every $\mathbb{F}_q$-member $X=\{F=0\}$ of $\mathcal{L}$, one of the following conditions holds:
\begin{enumerate}
    \item $X$ does not contain $Y$ (that is, $F$ is not divisible by $Q$), or
    \item $F$ has an $\mathbb{F}_q$-irreducible factor $G$ of degree $i-1$, where $G\in \langle G_1, \dots, G_{j} \rangle$.
\end{enumerate}

Now assume $\frac d2+1<i\leq d$ and $0\leq j\leq n$.
The property $\mathcal{P}_{i, j}$ implies property $\mathcal{P}_{Y, j}$, for if $F$ has an irreducible factor of degree $i$, then $F$ cannot have any factors of degree $d-i+1$ since $i>d-i+1$, so in particular $F$ cannot be divisible by $Q$. This shows that $\mathcal{P}_{Y,j}$ is a weaker property than $\mathcal{P}_{i, j}$, so the linear system $\mathcal{L}_0$ constructed in the proof of Theorem~\ref{thm:system-prescribed-factors} also satisfies $\mathcal{P}_{Y,j}$. A priori, a linear system satisfying $\mathcal{P}_{Y,j}$ could be larger; however, we will show that the same maximum dimension holds.

To show the sharpness, we proceed as in the proof of Theorem~\ref{thm:system-prescribed-factors}. We analogously define $\mathcal{R}_{Y, j}$ as the intersection
$$
\mathbb{P}(\{Q\cdot T \ | \ \deg(T)=i-1 \})\cap \mathcal{A},
$$
where $\mathcal{A}$ is the same as before. Assume there exists $E\in \mathcal{L}\cap \mathcal{R}_{Y,j}$ that is divisible by some $\mathbb{F}_q$-irreducible factor $G\in\langle G_1,\ldots,G_j\rangle$. Since we are now assuming the stricter condition $i>\frac{d}{2}+1$, we have $\deg(G) = i - 1 > d-i + 1 = \deg(Q)$, which ensures $\gcd(Q, G) = 1$. We then derive a contradiction as in the proof of Theorem~\ref{thm:system-prescribed-factors}.
\end{remark}

\begin{remark}\label{rmk:replace Fq} Theorem~\ref{thm:system-prescribed-factors} holds more generally if we replace the base field $\mathbb{F}_q$ with an arbitrary field $K$ that admits a separable extension of degree $\binom{n+i-1}{n}-j$. In particular, the result holds for all number fields. 
However, this does not hold for all fields $K$. For instance, if $K$ is algebraically closed, $j=0$, and $i=d$, then the maximal dimension is reduced by $n$; see \cite[Proposition 8.1]{AGR24}.
\end{remark}

\section{The proportion of points satisfying Theorem~\ref{thm:main}}\label{sec:conj}
We end the paper with some observations and conjectures on estimating $\mu(q,n,d,r)$.

Recall that in equation~\eqref{eq:mudef}, we defined the quantity $\mu(q,n,d,r)$, the proportion of points $P\in\bP^n(\FF_{q^r})$ for which the linear system of $F\in\mathcal{S}_{n,d}(\FF_q)$ vanishing at $P$ is as small as possible. Theorem \ref{thm:main} proves that $\mu(q,n,d,r)$ is always positive, but it is natural to ask whether more precise bounds or estimates are possible. 

\subsection{Large \texorpdfstring{$q$}{q}}

When $q$ is large, and $d$ is small compared to $q^r$, we find that $\mu=\mu(q,n,d,r)$ is close to $1$. 
\begin{proposition}
    Let $q$ be a prime power, and $n,d,r\geq 1$. Suppose $C\geq 1$ satisfies $d<Cq^{r/5}$. Then $\mu(q,n,d,r)>1-\frac{6C^5+3}{q}$.
\end{proposition}
\begin{proof}
    For $r=1$ we have $\mu=1$, so we may assume $r\geq 2$. For $n=1$, by Lemma~\ref{lem:n=1} we see that $\mu$ is bounded below by $\frac{1}{q^r+1}$ times the number of elements of $\FF_{q^r}$ that do not lie in any proper subfield of $\FF_{q^r}$. This implies
    \[\mu\geq \frac{q^r-\sum_{k=1}^{r-1}q^k}{q^r+1}=1-\frac{\frac{q^{r}-1}{q-1}}{q^r+1}>1-\frac{3}{q},\]
    so we may assume $n\geq 2$. For $d=1$ we apply Remark~\ref{rmk: simpler cases} to obtain $\mu\geq 1-\frac{1}{q-1}\geq 1-\frac{3}{q}$, so we may assume $d\geq 2$. Then we are in the setting of Proposition~\ref{prop:less than q}, with
    \[\Delta(q^r,d)< d^2q^{-r/2}+5 d^{13/3}q^{-r}< C^2 q^{-r/10}+5C^{13/3}q^{-2r/15}<6C^5.\]
    By Lemma~\ref{lem:bound t 2} we have $t(d-1)<2$, and we also have $q^{-|m-r|}\leq 1$, giving the desired bound.
\end{proof}
As a consequence, if we fix $C$, then $\mu(q,n,d,r)\to 1$ as $q\to\infty$ for any values of $n,d,r$ that satisfy the constraint $d<Cq^{r/5}$. For instance, if we take $C=3$, then as long as $r\geq d$ we always have $d<Cq^{r/5}$ and hence the proposition holds.

\subsection{\texorpdfstring{Large $n$ and $d$}{Large n and d}}

The situation is more interesting if we fix $q$ and allow the other parameters to vary. Here we consider the special case $r=m$: that is, given a prime power $q$, integers $n,d\geq 1$, and $m:=\binom{n+d}{d}$, what is the proportion of points $P\in\bP^n(\FF_{q^m})$ that do not lie on any degree $d$ hypersurface defined over $\FF_q$?
\begin{conjecture}\label{conj:qpochhammer}
    Fix a prime power $q$. Given $n,d\geq 1$, set $m=\binom{n+d}{n}$. As $n,d\to\infty$, the proportion of points of $\bP^n(\F_{q^m})$ that do not lie on any degree $d$ hypersurface over $\F_q$ converges to 
    \[\prod_{i=1}^\infty (1-q^{-i}).\]
\end{conjecture}

For instance, if $q=2$ this predicts that for sufficiently large values of $n$ and $d$, approximately $28.879\%$ of points in $\bP^n(\F_{2^m})$ do not lie on any degree $d$ hypersurface over $\F_2$. This has been observed experimentally in the course of checking exceptional cases (Appendix \ref{sec:computer}), even for $(n,d)$ as small as $(2,6)$: if $1000$ points are selected from $\bP^2(\F_{2^{28}})$ uniformly, typically between $270$ and $310$ of them do not lie on any degree $6$ hypersurface. For $q=3,4,5$ we obtain percentages around $56\%$, $69\%$, and $76\%$ respectively, with the limiting probability converging to $1$ as $q$ increases.

We now explain the heuristic behind Conjecture~\ref{conj:qpochhammer}.
After restricting to the affine chart $x_0=1$, a point
\[
P=(\alpha_1,\dots,\alpha_n)\in \bA^n(\F_{q^m})
\]
gives rise to the collection of monomial values
\[
\alpha_1^{j_1}\cdots \alpha_n^{j_n},
\qquad j_1+\cdots+j_n\le d.
\]
There are exactly $\binom{n+d}{n}=m$ such values. The condition that $P$ lies on no degree $d$ hypersurface is precisely that these $m$ elements of $\F_{q^m}$ be linearly independent over $\F_q$. For a randomly chosen point $P$, we might expect these monomial values to behave like $m$ random vectors in an $m$-dimensional $\F_q$-vector space, and thus the probability of linear independence should be close to
\[
\prod_{i=1}^m(1-q^{-i}).
\]

\appendix

\section{Exceptional cases}\label{sec:computer}

Recall that, by the discussion following Lemma~\ref{lem:reduce-n}, we may assume $n,d\geq 2$ and $r>\binom{n-1+d}{n-1}$.
In view of Section~\ref{subsec: checking inequality}, Theorem~\ref{thm:q 2 case}, Theorem~\ref{thm:q=d=2}, and Theorem~\ref{thm:inc-exc}, the following cases remain to be checked:
\begin{enumerate}[(i)]
    \item $q\leq 3$, $n=2$, $d\leq 6$, and $r\leq m+1$;
    \item $q=3$ and $r\leq 10$;
    \item $q=2$ and $r\leq 24$;
    \item $q=2$, $r=m$, $d\in\{3,4\}$ and $n\leq 35$;
    \item $q=2$, $r=m$, $d\in\{5,6\}$ and $n\leq 12$;
    \item $q=2$, $r=m$, $d\in\{7,8\}$ and $n\leq 10$;
    \item $q=2$, $r=m$, and $(n,d)=(9,9)$.
\end{enumerate}
We can check that the theorem holds in each of these cases by a finite computation. We have provided a GitHub repository that can be used to verify each of these cases \cite{code}. For each case $(q,n,d,r)$, the repository includes one example of a point $P\in \bP^n(\F_{q^r})$ for which the space of degree $d$ hypersurfaces through $P$ has the expected dimension (these points were found via random search). The full verification that the space of hypersurfaces through each of these points has the expected dimension took approximately $25$ minutes on a laptop.

We describe the method of verification here. 
The following Magma function takes a positive integer $d$, a finite field $F_0=\F_q$, and a non-zero tuple $P=(a_0,\ldots,a_n)$ consisting of elements $a_i$ in the field $F=\F_{q^r}$, and returns the dimension of the vector space $\{f\in\mathcal{S}_{n,d}(F_0) \ | \ f(P)=0\}$ over $F_0$.
\begin{verbatim}
function IncidentHypersurfaceDim(d, P, F0)
    n := #P - 1;
    allmonomials := [];
    for sub in Subsets({1..n+d}, n) do
        s := Sort(Setseq(sub));
        exp := [s[1]-1] cat [s[i+1]-s[i]-1 : i in [1..n-1]] 
               cat [n+d-s[n]]; 
        y := &*[P[i]^exp[i] : i in [1..n+1]];
        Append(~allmonomials, Eltseq(y, F0));
    end for;
    return Binomial(n+d,n) - Rank(Matrix(allmonomials));
end function;
\end{verbatim}

The code works as follows. First, recall that there is a one-to-one correspondence taking each $n$-element subset of $\{1,\ldots,n+d\}$ to a degree $d$ monic monomial in $n+1$ variables: if we label the elements of the subset in increasing order by $s_1,\ldots,s_n$, and set $s_0:=0$ and $s_{n+1}:=n+d+1$, then the corresponding monomial has each $x_i$ ($i=0,\ldots,n$) raised to the power of $s_{i+1}-s_i-1$. Given $P\in F^{n+1}$, we can compute a vector $(y_1,\ldots,y_m)\in F^m$ where $y_i$ is the $i$-th monomial evaluated at $P$.

Now any $f\in\mathcal{S}_{n,d}(F_0)$ corresponds to a linear form $\sum_{i=1}^m c_ix_i$ for some $c_1,\ldots,c_m\in F_0$, where $f(P)=\sum_{i=1}^m c_iy_i$. Applying the isomorphism $F\simeq F_0^r$ (implemented by \texttt{Eltseq(y, F0)} in the code above), we may replace each $y_i$ with $v_i\in F_0^r$. If we write $M\in M_{r\times m}(F_0)$ for the matrix with column vectors $v_1,\ldots,v_m$, then $f(P)=0$ if and only if the column vector $(c_1,\ldots,c_m)$ is in the kernel of $M$. Thus, the desired dimension equals $m- \operatorname{rank} M$. 

The main bottlenecks in the algorithm are computing all the monomials, and determining the rank of $M$. For simplicity, we consider the case $q=2$ and $r=m$, and assume $d$ is constant. There are $m$ monomials, and each can be computed using at most $d$ multiplications in $\F_{2^m}$. Multiplying two elements of $\F_{2^m}$ takes $O(m^2)$ bitwise operations, so this part of the algorithm takes $O(m^3)$ bitwise operations. Computing the rank of an $m\times m$ matrix using Gaussian elimination also takes $O(m^3)$ bitwise operations (though note that over $\F_2$ we have an advantage because no multiplications need to be performed). The finite fields involved in these special cases can get quite large, so an $O(m^3)$ algorithm takes a nontrivial amount of time. For instance, if $(n,d)=(35,4)$ then $r=m=82251$, so $m^3>10^{14}$; this explains why the verification takes over four minutes on a laptop to check this particular case.
That said, we have not devoted much effort to optimization, so it is possible that the verification code could be sped up further. 

\section*{Acknowledgments}
We thank Kaloyan Slavov for pointing out the reference \cite{S15}. The second author was supported by ERC Starting Grant 101076941 (`\textsc{Gagarin}'). The third author was supported in part by an NSERC fellowship.

\bibliographystyle{abbrv}
\bibliography{main}

\begin{thebibliography}{10}

\bibitem{AGR23}
S.~Asgarli, D.~Ghioca, and Z.~Reichstein.
\newblock Linear families of smooth hypersurfaces over finitely generated fields.
\newblock {\em Finite Fields Appl.}, 87:Paper No. 102169, 10, 2023.

\bibitem{AGR24}
S.~Asgarli, D.~Ghioca, and Z.~Reichstein.
\newblock Linear system of hypersurfaces passing through a {G}alois orbit.
\newblock {\em Res. Number Theory}, 10(4):Paper No. 84, 16, 2024.

\bibitem{AGY23}
S.~Asgarli, D.~Ghioca, and C.~H. Yip.
\newblock Existence of pencils with nonblocking hypersurfaces.
\newblock {\em Finite Fields Appl.}, 92:Paper No. 102283, 11, 2023.

\bibitem{code}
S.~Asgarli, J.~Love, and C.~H. Yip.
\newblock Hypersurfaces through a point: Github repository.
\newblock \url{https://github.com/jonathanrlove/hypersurfaces_through_point/}, October 30, 2024.

\bibitem{BK12}
A.~Bucur and K.~S. Kedlaya.
\newblock The probability that a complete intersection is smooth.
\newblock {\em J. Th\'eor. Nombres Bordeaux}, 24(3):541--556, 2012.

\bibitem{CM06}
A.~Cafure and G.~Matera.
\newblock Improved explicit estimates on the number of solutions of equations over a finite field.
\newblock {\em Finite Fields Appl.}, 12(2):155--185, 2006.

\bibitem{C16}
A.~Couvreur.
\newblock An upper bound on the number of rational points of arbitrary projective varieties over finite fields.
\newblock {\em Proc. Amer. Math. Soc.}, 144(9):3671--3685, 2016.

\bibitem{LW54}
S.~Lang and A.~Weil.
\newblock Number of points of varieties in finite fields.
\newblock {\em Amer. J. Math.}, 76:819--827, 1954.

\bibitem{P04}
B.~Poonen.
\newblock Bertini theorems over finite fields.
\newblock {\em Ann. of Math. (2)}, 160(3):1099--1127, 2004.

\bibitem{S91}
J.-P. Serre.
\newblock Lettre \`a{} {M}. {T}sfasman.
\newblock {\em Ast\'erisque}, 198-200:11, 351--353, 1991.
\newblock Journ\'ees Arithm\'etiques, 1989 (Luminy, 1989).

\bibitem{S15}
K.~Slavov.
\newblock The moduli space of hypersurfaces whose singular locus has high dimension.
\newblock {\em Math. Z.}, 279(1-2):139--162, 2015.

\bibitem{S94}
A.~B. {S\o rensen}.
\newblock On the number of rational points on codimension-{$1$} algebraic sets in {$\mathbb{P}^n(\mathbb{F}_q)$}.
\newblock {\em Discrete Math.}, 135(1-3):321--334, 1994.

\bibitem{S97}
M.~Sombra.
\newblock Bounds for the {H}ilbert function of polynomial ideals and for the degrees in the {N}ullstellensatz.
\newblock {\em J. Pure Appl. Algebra}, 117/118:565--599, 1997.
\newblock Algorithms for algebra (Eindhoven, 1996).

\end{thebibliography}
    
\end{document}